\documentclass{article}
\usepackage[margin=1.25in]{geometry}

\usepackage{amsmath}
\usepackage{amsthm}
\usepackage{amssymb}

\usepackage{tikz}
\usetikzlibrary{matrix}
\usepackage{tikz-cd}
\usetikzlibrary{cd}

\theoremstyle{definition}\newtheorem{teo}{Theorem}[section]
\theoremstyle{definition}\newtheorem{coro}[teo]{Corollary}
\theoremstyle{definition}\newtheorem{defi}[teo]{Definition}
\theoremstyle{definition}\newtheorem{lemma}[teo]{Lemma}
\theoremstyle{definition}\newtheorem{prop}[teo]{Proposition}
\theoremstyle{remark}\newtheorem{rmk}[teo]{Remark} 
\theoremstyle{remark}\newtheorem{strc}[teo]{}
\theoremstyle{remark}\newtheorem{exmp}[teo]{Example}
\theoremstyle{remark}

\newcommand{\pch}{pch(\mathbb{X})}
\newcommand{\chn}{ch(\mathbb{X})}
\newcommand{\C}{\mathcal{C}}
\newcommand{\X}{\mathbb{X}}

\newcommand{\A}{\mathbb{A}}
\newcommand{\Z}{\mathbb{Z}}

\newcommand{\T}{\mathcal{T}}
\newcommand{\F}{\mathcal{F}}
\newcommand{\M}{\mathcal{M}}

\usepackage{fancyhdr}
\fancypagestyle{firststyle}
{
   \fancyhf{}
   \fancyfoot[L]{\footnotesize \today}
   
}

\title{Torsion aspects of varieties of simplicial groups}
\author{Guillermo L\'opez Cafaggi}
\date{}

\begin{document}
\maketitle
\thispagestyle{firststyle}

\begin{abstract}
There is a lattice of torsion theories in simplicial groups such that the torsion/torsion-free categories are given by simplicial groups with truncated Moore complex below/above a certain degree. We study the restriction of these torsion theories to certain subcategories of simplicial groups. In particular, we prove that the categories of D.Conduch\'{e}'s 2-crossed modules and Ashley's crossed complexes in groups are semi-abelian and we give some descriptions of their torsion theories. These examples of torsion theories also give rise to new examples of pretorsion theories in the sense of A. Facchini and C. Finocchiaro, as well as examples of torsion torsion-free theories (TTF theories). 
\end{abstract}

\section{Introduction}

Semi-abelian categories \cite{JMT02} extend the validity of classical homological algebras to non-abelian  categories such as groups, Lie algebras, etc. Torsion theories, originally introduced for abelian categories, have been studied in semi-abelian categories (and other more general non-abelian settings) in \cite{BouGrn06} and \cite{CDT06}. However, we can notice that the modern applications of torsion theories in semi-abelian categories are very different from the classical results in categories of modules over rings. These new results are focused mainly on Categorical Galois Theory, factorization systems and homology (see \cite{GrnJan09}, \cite{EveGrn13}, \cite{EveGrn15}, among others). 

On the other side, one classical aspect of torsion theories in the abelian sense that has been overlooked in the semi-abelian case is that the torsion theories of a fixed category constitute a lattice. One reason that this lattice aspect has not been studied in-depth is that examples of torsion theories in semi-abelian categories are hard to find, and even harder to find several torsion theories for the same category. Perhaps, it is useful to recall a basic example. For a category $\X$ and torsion theories $(\mathcal{S}, \mathcal{G})$ and $(\T, \F)$ we say that   $ (\mathcal{S}, \mathcal{G}) \leq (\T, \F)$ if $ \mathcal{S} \leq \T$. So, for example in the category $Ab$ of abelian groups there are  torsion theories $(\mathbb{T}, \mathbb{F})$ and $(\mathbb{T}_p, \mathbb{F}_p)$ where $\mathbb{T}$ is the category of torsion abelian groups and $\mathbb{T}_p$ is the category of abelian groups whose elements have order a power of $p$ a prime. Thus, $(\mathbb{T}_p, \mathbb{F}_p) \leq (\mathbb{T}, \mathbb{F})$ for each prime $p$. 

One of the main interests of semi-abelian categories is the development of non-abelian homological algebra, the generalization of classical results to semi-abelian categories (for example in \cite{EveGrn10}, \cite{EveGrn15}, \cite{EveGrnLin}, \cite{EveLin}, among many others). Interesting results have been found in the category of internal groupoids (in a semi-abelian category), Whitehead's crossed modules, $n$-fold internal groupoids, etc. One first objective of this work, is to prove that some  categories that are mainly used as algebraic models of homotopy types are themselves semi-abelian, and thus expanding the previous list of semi-abelian examples. Namely, we will focus on the categories of D. Conduch\'{e}'s 2-crossed modules and N. Ashley's crossed complexes in groups, in fact we prove that these are Birkhoff subcategories/varieties of simplcial groups (epireflective subcategories closed under quotients). Our second objective is to study the torsion aspects of these new semi-abelian examples. In \cite{Lop22b}, a large family of torsion theories is introduced for simplicial groups. Since 2-crossed modules and crossed complexes are epireflective subcategories of simplicial groups, we can study the torsion theories therein. These restricted torsion theories present new properties not found in the previous examples of simplicial groups.

\subsection{Outline of the text}

Section 2 recalls the basic properties of torsion theories in semi-abelian categories as well as recalling the lattice $\mu(Grp)$ of torsion theories in simplicial groups which are defined in \cite{Lop22b} by truncation on the underlying Moore complexes.

Section 3 studies the lattice of torsion theories in the categories $\mathcal{M}_{n \geq}$ of simplicial groups with truncated Moore complex above degree $n$. It is well-known that $\mathcal{M}_{1 \geq}$ is equivalent to the category $Grpd(Grp)$ of internal groupoids in groups (see for example \cite{Lod82}). The category $Grpd(\X)$ of internal groupoids in $\X$, which is semi-abelian when $\X$ is so, exhibits two examples of torsion theories given by the pairs:
\[ (Ab(\X) , Eq(\X))\quad  \mbox{and} \quad (ConnGrpd(\X), Dis(\X)) \, .\] Where $Ab(\X)$, $Eq(\X)$, $ConnGrpd(\X)$ and $Dis(\X)$ are, respectively, the subcategories of  $Grpd(\X)$ of internal abelian groups, internal equivalence relations, connected groupoids and discrete internal groupoids (these examples were already studied in \cite{BouGrn06}, \cite{EveGrn10}). Here, it is proved that these known examples in internal groupoids are induced by the lattice $\mu(Grp)$ of simplicial groups restricted to the subcategory $\mathcal{M}_{1 \geq}$. 

Also it is a classical result (again from \cite{Lod82}) that the category of internal groupoids in groups is equivalent to the category $\X Mod$ of Whitehead's crossed modules, thus the torsion theories in $Grpd(Grp)$ correspond to the torsion theories in $\X Mod$: 
\[\mbox{(abelian groups, inclusion of normal monomorphisms)}\] and \[ \mbox{(central extensions in groups, groups as discrete crossed modules)}.\]
 Similarly, in \cite{Con84}, D. Conduch\'{e} introduces the category ${}_2\X Mod$ of 2-crossed modules and proves that is equivalent to $\mathcal{M}_{2 \geq}$. Since the categories $\mathcal{M}_{n \geq}$ are semi-abelian, the category ${}_2\X Mod$ is semi-abelian. Torsion theories in $\mathcal{M}_{2 \geq}$ and $\mathcal{M}_{n \geq}$ are studied in Theorems \ref{teom2} and \ref{teomn}.

In section 4, we study pretorsion theories associated to the torsion theories in $\mu(Grp)$. A pretorsion theory $(\T, \F)$, introduced in \cite{FacFin}, is a generalization of a torsion theory. Specially used for non-pointed categories, it relies roughly in replacing the zero object in the definition of a torsion theory for a category $\mathcal{Z}$ of \textit{trivial} objects. Here, we proved that given torsion theories $(\mathcal{S}, \mathcal{G}) \leq (\T, \F)$ we have a pretorsion theory $(\T, \mathcal{G})$ and $\mathcal{Z}= \T \cap \mathcal{G}$ (this was first noticed in \cite{Man15}). Applied to simplicial groups we can find examples of pretorsion theories where the categories of trivial objects consists of either Eilenberg-Mac Lane simplicial groups or simplicial groups with trivial homotopy groups.

In section 5 and 6, the category $Crs(Grp)$ of Ashley's reduced crossed complexes (or crossed complexes in groups) is proved to be a semi-abelian category. An important feature of torsion theories  in reduced crossed modules is that they behave like a weak TTF theory. Introduced in \cite{jans65} for abelian categories, a torsion torsion-free theory (or TTF theory for short) in a abelian category $\X$ is a triplet of full subcategories $(\C, \T, \F)$ of $\X$ such that $(\C, \T)$ and $(\T, \F)$ are torsion theories in $\X$, so $\T$ is called a torsion torsion-free category. The subcategory $Crs(Grp)_{n \geq}=Crs(Grp) \cap \mathcal{M}_{n \geq}$ of $Crs(Grp)$ is not a torsion torsion-free subcategory but is a torsion-free coreflective subcategory, so we introduce the notion of a CTF theory. For abelian categories, we have that TTF theories and CTF theories are the same. On the other hand, $Crs(Grp)$ present examples of pairs of subcategories that satisfy the axioms of a torsion theory only relative to a particular class of objects $\mathcal{E}$ which we will call $\mathcal{E}$-torsion theories, see Theorem \ref{teocrs}.

\section{Torsion theories in simplicial groups with truncated Moore complex}

A category $\X$ is called semi-abelian \cite{JMT02} if it is pointed with binary coproducts, Barr exact and Bourn protomodular. A finitely complete category $\X$ is called regular if it has all coequalizers of kernel pairs and these are pullback stable, furthermore, a regular category is called (Barr)-exact if every internal equivalence relation is the kernel pair of some arrow in $\X$. On the other hand, a pointed category $\X$ is Bourn protomodular (\cite{Bou91}) if and only if the short split five lemma holds in $\X$:

Given a commutative diagram in $\X$
\[\begin{tikzcd}  0\ar[r]& K\ar[r, "ker(f)"]\ar[d, "\alpha"]& X\ar[r, "f"]\ar[d, "\beta"] &Y\ar[r] \ar[d, "\gamma"]  & 0 \\0\ar[r] & K'\ar[r, "ker(f')"'] & X'\ar[r,"f'"'] &Y'\ar[r] &0 \end{tikzcd} \]
with $f$ and $f'$ split epimorphisms, then if $\alpha$ and $\gamma$ are isomorphism then so is 
$\beta$.

\begin{defi} A torsion theory in a semi-abelian category $\X$ is a pair $(\T, \F)$ of full and replete subcategories of $\X$ such that:
\begin{enumerate}
\item[TT1] A morphism $f:T \to F$ with $T$ in $\T$ and $F$ in $\F$ is a zero morphism.
\item[TT2] For any object $X$ in $\X$ there is a short exact sequence:
\begin{equation}\label{sestt}
\begin{tikzcd} 0\ar[r] &T_X\ar[r, "\epsilon_X"] & X \ar[r, "\eta_X"] & F_X \ar[r]& 0\end{tikzcd} 
\end{equation} 
with $T_X$ in $\T$ and $F_X$ in $\F$. 
\end{enumerate}
\end{defi}

In a torsion theory $(\T, \F)$, $\T$ is called the torsion category and its objects are called torsion objects. In a similar way, $\F$ is the torsion-free category. Any subcategory is called torsion, resp. torsion-free, if it is the torsion category, resp. torsion-free, of a torsion theory. The torsion category $\T$ is a normal mono-coreflective subcategory of $\X$, i.e., the inclusion $J:\T \to \X$ has a right adjoint $T: \X \to \T$ and each component of the counit $\epsilon_X: JT(X) \to X$ is a normal monomorphism (a kernel of some arrow in $\X$). Similarly, $\F$ is a normal epi-reflective subcategory of $\X$, i.e., the inclusion $I$ has a left adjoint $F$ and each component $\eta_X:X \to IF(X)$ is a normal epimorphism: 	

\begin{equation}\label{ttadj}
\begin{tikzcd}[column sep=large] \T \ar[r, bend left, "J"]\ar[r, phantom, "\perp"] & \X \ar[r, bend left, "F"]\ar[l, bend left,"T"]\ar[r, phantom, "\perp"] & \F \, .\ar[l, bend left, "I"] \end{tikzcd}
\end{equation}

Then, it is easy to observe that $\T$ is closed under colimts in $\X$ (those that exist in $\X$) and $\F$ is closed under limits in $\X$. Both $\T$ and $\F$ are closed under extensions in $\X$, this means that given a short exact sequence in $\X$:
\[\begin{tikzcd} 0\ar[r] & A\ar[r] & X\ar[r] & B\ar[r]& 0 \end{tikzcd}\]
with $A$ and $B$ in $\T$ (resp. in $\F$) then $X$ is in  $\T$ (resp. in $\F$). A torsion theory $(\T, \F)$ is called \textit{hereditary} if $\T$ is closed under subobjects in $\X$, i.e., given a monomorphism $m:M \to T$ with $T$ in $\T$ then $M$ is also torsion. Conversely, a torsion theory is called \textit{cohereditary} if $\F$ is closed under quotients in $\X$, i.e., given a normal epimorphism $q: F \to Q$ with $F$ in $\F$ then $Q$ is also in $\F$. It is useful to recall the following result. 

\begin{prop}\label{semiab} \cite{Lop22b} Let $(\T, \F)$ be a torsion theory of a semi-abelian category $\X$, if $(\T,\F)$ is hereditary then $\T$ is a semi-abelian category. Accordingly, if $(\T, \F)$ is cohereditary then $\F$ is a semi-abelian category.
\end{prop}

Given two torsion theories $(\T, \F)$ and $(\mathcal{S}, \mathcal{G})$ in $\X$, we have that $\T \subseteq \mathcal{S}$ if and only if $\mathcal{G} \subseteq \F$. This allows one to introduce a partial order in the class of torsion theories in a category $\X$ by $(\T, \F) \leq (\mathcal{S}, \mathcal{G})$ if $\T \subseteq \mathcal{S}$. There is a bottom element and a top element given by the trivial torsion theories $0:= (0, \X)$ and $\X:=(\X,0)$.

We recall the basic definitions of simplicial groups and the Moore normalization, we refer to \cite{May67} for the details.

 The simplicial category $\Delta$ has, as objects, finite ordinals $[n]=\{0,1,\dots n\}$ and as morphisms monotone functions. A simplicial group is a functor $X: \Delta^{op} \to Grp$, we will denote $Simp(Grp)$ for the category of simplicial groups.  A simplicial group $X$ can be equivalently defined by the following data:
 a family of objects $\{X_n\}_{n\in \mathbb{N}}$ in $\X$, the \textit{face morphisms} $d_i: X_n\to X_{n-1}$ and the \textit{degeneracies morphisms} $s_i: X_n \to X_{n+1}$:
\[X= \begin{tikzcd}
\dots X_n  \ar[r, "\vdots", phantom, shift left]
       \ar[r, "d_0"', shift right=8]
        \ar[r, "d_n", shift left=8]  
 & \dots \ar[r, "\vdots", phantom, shift left]
       \ar[r, "d_0"', shift right=8]
        \ar[r, "d_4", shift left=8]
         \ar[l, "s_{n-1}"', shift right=4, shorten=1.5mm ]
       \ar[l,"s_0", shift left=4, shorten=1.5mm ]
 & X_3  \ar[r,"d_0"' ,shift right=12] 
       \ar[r, , "d_1"', shift right=4] 
       \ar[r, "d_2"' ,shift left=4] 
       \ar[r, "d_3",shift left=12] 
       \ar[l, "s_3"', shift right=4, shorten=1.5mm ]
       \ar[l,"s_0", shift left=4, shorten=1.5mm ]
 & X_2 \ar[r, "d_1"']
        \ar[r, "d_2", shift left=8] 
        \ar[r, "d_0"', shift right=8]
        \ar[l, "s_1", shorten=1.5mm]
        \ar[l, "s_2"', shorten=1.5mm,shift right=8] 
        \ar[l, "s_0", shorten=1.5mm, shift left=8] 
 & X_1 \ar[r, "d_1", shift left=4] 
       \ar[r, "d_0"', shift right=4] 
       \ar[l,"s_1"', shift right=4, shorten=1.5mm ] 
       \ar[l,"s_0", shorten=1.5mm  ,shift left=4 ]
 & X_0 \ar[l, "s_0", shorten=1.5mm]
\end{tikzcd}\] 
 satisfying the \textit{simplicial identities}:
\begin{align*}
 d_i d_j &= d_{j-1} d_i   \quad   \mbox{if} \quad i<j  \\ 
 s_i s_j &= s_{j+1} s_i   \quad  \mbox{if} \quad i\leq j \\
 d_i s_j &= 
 \left\{\begin{array}{cc}
  s_{j-1} d_i & \mbox{if} \quad i<j  \\
 1  & \mbox{if} \quad i=j\, \mbox{or}\,\ i=j+1 \\
  s_j d_{i-1} & \mbox{if} \quad i>j+1\, .
 \end{array}\right. 
\end{align*} 
Since $Simp(Grp)$ is a functor category with codomain a semi-abelian category, it is semi-abelian.

\begin{strc} An $n$-truncated simplicial group $X$ is a simplicial group where the objects $X_i$, the face morphisms $d_i$ and degeneracies morphisms $s_i$ are defined up to level $n$ and they satisfy the simplicial identities (those that make sense). Let $Simp_n(Grp)$ be the category of $n$-truncated simplicial groups then there is a truncation functor:
\[tr_n: \begin{tikzcd}  Simp(Grp)\ar[r] &  Simp_n(Grp) \end{tikzcd}\]
which simply forgets everything above degree $n$. For all $n$, the functor $tr_n$ has a left adjoint $sk_n$ called the $n$-skeleton, and a right adjoint $cosk_n$ named the $n$-coskeleton: $sk_n \dashv tr_n \dashv cosk_n$. We will write $Sk_n=sk_ntr_n$ and $Cosk_n=cosk_n tr_n$.

In addition, for $n=0$, the functor $sk_0$ has a left adjoint $\pi_0$ where $\pi_0(X)=coeq(d_0,d_1)$ is the coequalizer of the face morphisms $d_0,d_1:X_1\to X_0$. Moreover, the adjunctions $\pi_0 \dashv sk_0 \dashv tr_0 \dashv cosk_n$ correspond to the functors $\pi_0 \dashv Dis \dashv ()_0 \dashv Ind$. For a group $G$, $Dis(G)$ is the discrete simplicial group of $G$ 
\[\begin{tikzcd} 
Dis(G)=  & \dots\ar[r, shift left=3, "1"]\ar[r, shift right=3, "1"']\ar[r, phantom, "\vdots", shift left] 
& G\ar[r, shift left=3, "1"]\ar[r, shift right=3, "1"']\ar[r, phantom, "\vdots", shift left]  
& G\ar[r, shift left=3, "1"]\ar[r, shift right=3, "1"'] \ar[r, phantom, "\vdots", shift left] 
& G\ar[r, shift left=3, "1"]\ar[r, shift right=3, "1"']
& G\ar[l, "1"] \end{tikzcd}\]
and $Ind(G)$ is the indiscrete simplicial group of $G$
\[\begin{tikzcd} 
Ind(G)=  & \dots\ar[r, shift left=3, "p_0"]\ar[r, shift right=3, "p_4"']\ar[r, phantom, "\vdots", shift left] 
& G^4\ar[r, shift left=3, "p_0"]\ar[r, shift right=3, "p_3"']\ar[r, phantom, "\vdots", shift left]  
& G^3\ar[r, shift left=3, "p_0"]\ar[r, shift right=3, "p_2"'] \ar[r, phantom, "\vdots", shift left] 
& G^2\ar[r, shift left=3, "p_0"]\ar[r, shift right=3, "p_1"']
& G\ar[l, "s_0"] \end{tikzcd}\]
where $G^n$ is the $n$-fold product, $(X)_0=X_0$ and the morphisms $p_i$ are induced by product projections.
\end{strc}

A chain complex $M$ is a family of morphisms $\{\delta_n: M_n \to M_{n-1}\}_{n \in \mathbb{N}}$ such that $\delta_n \delta_{n+1}=0$ for all $n$. A chain complex $M$ is \textit{proper} if for each differential $\delta_n$ the monomorphism $m_n$  of the normal epi/mono factorization $(e_n, m_n)$ of $\delta_n$ is a normal monomorphism:
\[\begin{tikzcd} \dots\ar[r] &  M_{n+1}\ar[r, "\delta_{n+1}"]& M_n\ar[rr, "\delta_n"]\ar[rd, "e_n"', two heads] & &  M_{n-1}\ar[r, "\delta_{n-1}"]& \dots\, . \\ &&& \delta_n(M_n)\ar[ru, "m_n"', hook]&& \end{tikzcd}\]
The category of chain complexes and the subcategory of proper chain complexes in groups will be denoted as $chn(Grp)$ and $pch(Grp)$.  

\begin{defi} (see \cite{May67}) The Moore normalization functor $N: Simp(Grp) \to chn(Grp)$ is defined as follows. Let $X$ be a simplicial group then $N(X)$ is the group chain complex
\[\begin{tikzcd}  \dots\ar[r] & N(X)_n\ar[r, "\delta_n"] & N(X)_{n-1}\ar[r]& \dots\ar[r] & N(X)_1 \ar[r, "\delta_1"] & N(X)_0  \end{tikzcd} \]
 such that $N(X)_0=X_0$ and 
\[N(X)_n = \bigcap_{i=0}^{n-1} ker(d_i: X_n \to X_{x-1})\]
and differentials $\delta_n= d_n \circ \cap_i ker(d_i): N(X)_n \to N(X)_{n-1}$ for $n \geq 1$.  

The Moore chain complex $N(X)$ of a simplicial group is proper. It is known that the functor $N$ preserves finite limits and normal epimorphisms, and also it is also conservative (see, for example, \cite{EveLin}).

We will write $\M_{n\geq}$ for the subcategory of simplicial groups with trivial Moore complex above degree $n$. Similarly, $\M_{\geq n}$ is the subcategory of simplicial groups with trivial Moore complex below degree $n$. 
\end{defi}

In \cite{Lop22b} it is proved that $\M_{n \geq}$ is a torsion-free subcategory of $Simp(Grp)$, and respectively, $\M_{\geq n}$ is a torsion subcategory for all $n$. Indeed, the corresponding torsion theories form a linearly order lattice $\mu(Grp)$ in $Simp(Grp)$. We recall how these are defined and useful properties.

\begin{teo}\label{teolop} \cite{Lop22b} There is a linear order lattice  $\mu(Grp)$ of torsion theories in $Simp(Grp)$:
\begin{multline*}
0  \leq \dots \leq \; \mu_{ n+1\geq}  \;  \leq \;   \mu_{\geq n+1} \;  \leq \; \mu_{n\geq} \;   \leq \;  \mu_{\geq n} \;  \leq \dots \\  \dots \leq \; \mu_{\geq 2} \; \leq  \; \mu_{1\geq} \; \leq \;  \mu_{\geq 1} \;  \leq  \; \mu_{0\geq} \;   \leq Simp(Grp)\, . 
\end{multline*}
where
\begin{enumerate}
\item The torsion theory $\mu_{n \geq}$ is given by the pair $(Ker(Cot_n), \M_{n \geq})$, where $Cot_n:Simp(Grp) \to \M_{n \geq}$ (introduced in \cite{Por93}) is the left adjoint of the inclusion $i:\M_{n \geq} \to Simp(Grp)$  and $Ker(Cot_n)$ is the full subcategory of simplicial groups $X$ such that $Cot_n(X)=0$.
\item The torsion theory $\mu_{\geq n}$ is given by the pair $(\M_{\geq n}, \F_{tr_{n-1}})$, where $\F_{tr_{n-1}}$ is the subcategory of simplicial groups  $X$ with $\eta_X$ monic where $\eta$ is the unit of the adjunction $tr_{n-1} \dashv cosk_{n-1}$. 
\item Let $X$ be a simplicial group and $M$ its Moore complex. For $X$ the associated short exact sequence of $\mu_{n \geq}$  under normalization is the short exact sequence  in chain complexes (written vertically):
\begin{equation}\label{sescok}
\begin{tikzcd}
\hdots \ar[r]
 & M_{n+2} \ar[r, "\delta_{n+2}"] \ar[d, "id"]
 & M_{n+1} \ar[r, "e_{n+1}", two heads] \ar[d, "id"]
  & \delta_{n+1}(X_{n+1}) \ar[r] \ar[d, "m_{n+1}", hook]
   & 0 \ar[r] \ar[d]
    & \hdots
 \\ 
\hdots \ar[r]
 & M_{n+2} \ar[r, "\delta_{n+2}"] \ar[d]
 & M_{n+1} \ar[r, "\delta_{n+1}"] \ar[d]
  & M_{n} \ar[r, "\delta_{n}"] \ar[d, "cok(\delta_{n+1})", two heads] 
   & M_{n-1} \ar[r] \ar[d, "id"]
    & \hdots
\\
\hdots \ar[r]
 & 0 \ar[r]
 & 0 \ar[r] 
  & Cok(\delta_{n+1}) \ar[r]
   & M_{n-1} \ar[r]
    & \hdots
\end{tikzcd}
\end{equation}
\item Let $X$ be a simplicial group and $M$ its Moore complex. For $X$ the associated short exact sequence of $\mu_{ \geq n}$ under normalization is the short exact sequence  in chain complexes (written vertically):
\begin{equation}\label{sesker}
\begin{tikzcd}
\hdots \ar[r]
 & M_{n+1} \ar[r] \ar[d, "id"]
  & Ker(\delta_n) \ar[r] \ar[d, "ker(\delta_n)", hook]
   & 0 \ar[r] \ar[d]
   & 0 \ar[r] \ar[d]
    & \hdots
 \\ 
\hdots \ar[r]
 & M_{n+1}\ar[r, "\delta_{n+1}"] \ar[d]
  & M_n \ar[r, "\delta_n"] \ar[d, "e_n", two heads] 
   & M_{n-1} \ar[r, "\delta_{n-1}"] \ar[d, "id"]
    & M_{n-2} \ar[r] \ar[d, "id"]
    & \hdots
\\
\hdots \ar[r]
 & 0 \ar[r] 
  & \delta(M_n)\ar[r, "m_n"]
   & M_{n-1} \ar[r]
    & M_{n-2} \ar[r] 
    & \hdots
\end{tikzcd}
\end{equation}
\end{enumerate}
\end{teo}

As shown from diagrams (\ref{sescok}) and (\ref{sesker}) the coreflector of $\M_{\geq n}$ and respectively the reflector of $\M_{n\geq}$ behave, at the level of Moore complexes, as the truncation functors introduced by L. Illusie in \cite{ill71}. As a consequence, torsion/torsion-free objects can be characterized by their Moore complex as follows.   

\begin{coro}\label{cotnsurj} \cite{Lop22b} Let $X$ be a simplicial group  with Moore complex $M$, then
\begin{enumerate} 
\item  $X$ belongs to $Ker(Cot_n)$ if and only if $M$ is of the form
\[\begin{tikzcd}M=\dots\ar[r]& M_{n+2}\ar[r]& M_{n+1}\ar[r, "\delta_n", two heads] & M_n\ar[r]& 0\ar[r]& 0\ar[r]& \dots \end{tikzcd}\]
with $\delta_n$ a normal epimorphism.
\item  $X$ belongs to $\F_{tr_n}$ if and only if $M$ is of the form
\[\begin{tikzcd}M=\dots\ar[r] & 0\ar[r] & 0\ar[r]& M_{n+1}\ar[r, "\delta_{n+1}", hook]& M_n\ar[r]& M_{n-1}\ar[r] &\dots \end{tikzcd}\]
with $\delta_{n+1}$ a normal monomorphism.
\end{enumerate}
\end{coro} 

\section{Torsion subcategories of $\M_{n \geq}$}

The lattice $\mu(Grp)$
\[ 0  \leq \quad  \dots \leq \quad \mu_{\geq 2} \quad \leq  \quad \mu_{1\geq} \quad \leq \quad  \mu_{\geq 1} \quad  \leq  \quad \mu_{0\geq} \quad   \leq Simp(Grp)\, . \]
extends the torsion theories in internal groupoids $Grpd(Grp)$
\[0=(0, Grpd(Grp)) \leq \; (Ab,Eq(Grp) ) \; \leq \; (ConnGrpd(Grp) , Dis(Grp)) \; \leq \; Grpd(Grp)\]
in the sense that the torsion-free categories of the first 3 largest non-trivial torsion theories,   $ \mu_{1 \geq} \leq \mu_{\geq 1} \leq \mu_{0 \geq}$, in $Simp(Grp)$ are the torsion-free categories ($Grpd(Grp)$, $Eq(Grp)$ and $Dis(Grp)$, respectively) of $Grpd(Grp)$.

On the other hand, the torsion categories of $Simp(Grp)$ are not as easily described as in the case of $Grpd(Grp)$. However, some similarities arise when we restrict the lattice $\mu(Grp)$ to subcategories of the form $\M_{n \geq}$ with the case $Grpd(Grp)= \M_{1 \geq}$. To this end, we recall how simplicial groups are equivalent to chain complexes with operations following the work of J. Loday in \cite{Lod82}, D. Conduch\'e in \cite{Con84} and, P. Carrasco and A. M. Cegarra in \cite{CaCe91}.

\subsection{The case of $\M_{1 \geq}$ and crossed modules}
For completeness sake we quickly recall how torsion theories in internal groupoids correspond to torsion theories in equivalent category of Whitehead's crossed modules.

We will write group actions acting on the left ${}^b(a)$ and each group $G$ is consider to act on itself by conjugation as ${}^g(g')=gg'g^{-1}$.

\begin{defi}\label{deficroos} \cite{Wht} A \textit{crossed module} (in groups) is a morphism of groups $\delta: A \to B$ with a groups action  $B \to Aut(A)$ such that:
\begin{enumerate}
\item $\delta({}^b(a))=b\delta(a)b^{-1} $  ($\delta$ is equivariant). 
\item $ {}^{\delta(a)a'}=aa'a^{-1}$  (Peiffer identity).
\end{enumerate}
If $\delta$ only satisfies axiom 1 then it is called a \textit{precrossed module}.  We will denote the category of crossed modules and of precrossed modules as $\X Mod$ and $P \X Mod$, respectively. 
\end{defi} 
 
From \cite{Lod82}, the following categories are equivalent:
\begin{enumerate}
\item The category $\M_{1 \geq}$ of simplicial groups with trivial Moore complex above degree 1.
\item The category $Grpd(Grp)$ of internal groupoids in groups.
\item The category $\X Mod$ of crossed modules in groups.
\item The category of \textit{Cat-1-groups}.
\end{enumerate} 
Indeed, given a simplicial group $X$ with Moore complex $M=N(X)$, the differential $\delta_1: M_1 \to M_0$ is a precrossed module with the action given by conjugation with $s_0$. Furthermore, $\delta_1$ is a crossed module if and only if $M_i=0$ for $i>1$. 

We will write $\mu(\M_{1 \geq})$ for the lattice given by the torsion theories of $\mu(Grp)$ restricted to $\M_{1 \geq}$, i.e., we consider the torsion theory $(\T \cap \M_{1 \geq},\F\cap \M_{1 \geq})$ for each element $(\T, \F)$ of $\mu(Grp)$. We will write $\mu'_{n \geq}$ and $\mu'_{\geq n}$ for the corresponding restrictions of $\mu_{n \geq}$ and $\mu_{\geq n}$.

\begin{prop}\label{latm1} The lattice $\mu(\M_{1\geq})$ is given by
\[0  \quad \leq  \quad \mu'_{\geq 1}  \quad \leq  \quad \mu'_{0 \geq} \quad \leq \quad \M_{1\geq} \, . \]
\end{prop}
\begin{proof} Recall that $\mu(Grp)$ is given by
\begin{equation}\label{mugrpdiag}
\begin{tikzcd}[column sep=scriptsize]
Simp(Grp)= & Simp(Grp) \ar[r, shift left=2.5]\ar[r, phantom, "\perp"] 
  & Simp(Grp) \ar[r, shift left=2.5]\ar[l, shift left=2]\ar[r, phantom, "\perp"] 
   & 0  \ar[l, shift left=2] \ar[d, hook] \\
\mu_{0\geq}= & Ker(Cot_0) \ar[r, shift left=2.5]\ar[r, phantom, "\perp"]\ar[u, hook] 
  & Simp(Grp) \ar[r, shift left=2.5]\ar[l, shift left=2.5]\ar[r, phantom, "\perp"] 
   & \mathcal{M}_{0 \geq} \cong Grp  \ar[l, shift left=2.5] \ar[d, hook] \\
\mu_{\geq 1}= & \mathcal{M}_{\geq 1} \ar[r, shift left=2.5]\ar[r, phantom, "\perp"] \ar[u, hook]
  & Simp(Grp) \ar[r, shift left=2.5]\ar[l, shift left=2.5]\ar[r, phantom, "\perp"] 
   & \F_{tr_0} \cong Eq(Grp)  \ar[l, shift left=2.5] \ar[d, hook]\\
\mu_{1\geq}= & Ker(Cot_1) \ar[r, shift left=2.5]\ar[r, phantom, "\perp"]\ar[u, hook] 
  & Simp(Grp) \ar[r, shift left=2.5]\ar[l, shift left=2.5]\ar[r, phantom, "\perp"] 
   & \mathcal{M}_{1 \geq} \cong Grpd(Grp)  \ar[l, shift left=2.5] \ar[d, hook] \\
\mu_{\geq 2}= & \mathcal{M}_{\geq 2} \ar[r, shift left=2.5]\ar[r, phantom, "\perp"] \ar[u, hook]
  & Simp(Grp) \ar[r, shift left=2.5]\ar[l, shift left=2.5]\ar[r, phantom, "\perp"] 
   & \F_{tr_1} \ar[l, shift left=2.5] \ar[d, hook] \\
\mu_{2\geq}= & Ker(Cot_2) \ar[r, shift left=2.5]\ar[r, phantom, "\perp"]\ar[u, hook] 
  & Simp(Grp) \ar[r, shift left=2.5]\ar[l, shift left=2.5]\ar[r, phantom, "\perp"] 
   & \mathcal{M}_{2 \geq}  \ar[l, shift left=2.5] \\
\dots &  \dots & \dots &  \dots 
\end{tikzcd}
\end{equation}
Since $\M_{1 \geq}$ is itself a torsion-free subcategory comprised in the lattice $\mu(Grp)$, the restriction of each torsion theory below $\mu_{1 \geq}$ is the trivial torsion theory $(0, \M_{1 \geq})$ in $\M_{1 \geq}$.
\end{proof}

It is easy to observe that the torsion-free categories $Eq(Grp)$ and $Dis(Grp)$ in $Simp(Grp)$ are also torsion-free categories of $\M_{1 \geq}$. We can conclude the following.

\begin{prop} The lattice $\mu(\M_{1 \geq})$ corresponds to 
\[0 \quad \leq \quad (Ab(Grp), Eq(Grp)) \quad \leq \quad (ConnGrpd(Grp), Dis(Grp)) \quad \leq \quad \M_{1 \geq}. \]
Moreover, under the Moore normalization this lattice corresponds to the lattice of torsion theories in $\X Mod$
\[0 \quad \leq \quad (Ab, Norm) \quad \leq \quad (CExt, Dis) \quad \leq \quad \X Mod\, ,\]
where 
\begin{enumerate}
\item $Ab$ is the category of crossed modules of the form $A \to 0$ for an abelian group $A$.
\item $Norm$ is the category of crossed modules given by the inclusion of a normal subgroup  $i: N \to G$.
\item $CExt$ is the category of crossed modules given by central extensions, epimorphisms $p: G \to Q$ with a $ker(p) \leq Z(G)$ where $Z(G)$ is the center of $G$.
\item $Dis$ is the category of crossed modules given of the form $0 \to G$ for a group $G$. 
\end{enumerate}
\end{prop}
\begin{proof} In a torsion theory the torsion and the torsion-free category uniquely determine each other, so if $Eq(Grp)$ and $Dis(Grp)$ are torsion-free categories of $\mu'_{0 \geq}$ and $\mu'_{\geq 1}$ the torsion categories must correspond accordingly to $Ab(Grp)$ and $ConnGprd(Grp)$.

The second statement is well-known, for instance the equivalences are mentioned in \cite{BouGrn06}, \cite{EveGrn10} and  \cite{Man15}. 
\end{proof}

\subsection{The case of $\M_{n \geq}$}

Following the previous case, we will write $\mu(\M_{n \geq})$ for the restriction of $\mu(Grp)$ to $\M_{n \geq}$ and $\mu'_{i \geq}$, $\mu'_{\geq i}$ for the restrictions of the torsion theories $\mu_{ i \geq}$, $\mu_{\geq i}$.

Similar to Proposition \ref{latm1}, also from the diagram (\ref{mugrpdiag}) we have the following result.

\begin{prop}\label{latmn} The lattice $\mu(\M_{n\geq})$ is given by
\[0 \quad \leq \quad  \mu'_{\geq n} \quad \leq \quad  \mu'_{n-1 \geq}  \quad \leq  \quad \mu'_{\geq n-1} \quad \leq \dots \leq  \quad \mu'_{\geq 1}  \quad \leq  \quad \mu'_{0 \geq} \quad \leq \quad \M_{n\geq}. \]
\end{prop}

Just as in $\M_{1 \geq}$ the subcategories $Dis(Grp)$, $Eq(Grp)$ and $Grpd(Grp)$ are torsion-free subcategories of $\M_{n \geq}$. In order to characterise the torsion categories  we recall some categories introduced by D. Conduch\'e.

\begin{defi} \cite{Con84} A group chain complex
\[\begin{tikzcd} L\ar[r, "\delta_2"] & M\ar[r, "\delta_1"] & N \end{tikzcd} \]
is called a 2-\textit{crossed module} if $N$ acts on $L$ and $M$ and the differentials $\delta_2, \delta_1$ are equivariant ($N$ acts over itself with conjugation), and there is a mapping
\[\{\ , \ \}: \begin{tikzcd} M \times M \ar[r]& L \end{tikzcd}\]
satisfying:
\begin{itemize}
\item[2XM1] $\delta_2\{m_0 , m_1\}= m_0m_1m_0^{-1} \ ^{\delta_1(m_0)}(m_1^{-1})$;
\item[2XM2] $\{ \delta_2(l_0), \delta_2(l_1) \}=[l_0 , l_1]$;
\item[2XM3] $\{\delta_2(l), m\}\{m, \delta_2(l)\}=l\ ^{\delta_1(m)}l^{-1}$;
\item[2XM4] $\{m_0,m_1m_2\}=\{m_0, m_1\}\{m_0,m_2\}\{\delta_2\{m_0,m_2\}^{-1},\ ^{\delta_1(m_0)}m_1 \}$;
\item[2XM5] $\{m_0m_1,m_2 \}=\{m_0,m_1m_2m_1^{-1}\}\ ^{\delta_1(m_0)}\{m_1,m_2\}; $
\item[2XM6] $ {}^n\{m_0,m_1\}= \{ {}^n m_0,  {}^nm_1\}$.
\end{itemize} 
The map $\{\ , \ \}$ is called the \textit{Peiffer lifting}. 

A morphism of 2-crossed modules is a morphism of chain complexes that preserves the group actions and the Peiffer lifting. The category of 2-crossed complexes will be denoted as $_2\X Mod$.
\end{defi}

In a 2-crossed module the Peiffer lifting defines an action of $M$ over $L$ as ${}^m(l)=l\{\delta_2(l)^{-1} ,m\}$, so $\delta_2$ is indeed a crossed module. On the other hand, $\delta_1$ is only a precrossed module. A crossed module is a 2-crossed module by setting $L=0$. If a 2-crossed module has $N=0$ then the equations get simplified, thus we obtain a reduced 2-crossed module as follows.  

\begin{defi}\label{redcrossdefi} \cite{Con84}
 A \textit{reduced 2-crossed module} is a group morphism $\delta: L \to M$ with a map $\{\ , \ \}: M\times M \to L$ satisfying:
\begin{enumerate}
\item $\delta \{m_0,m_1\}=[m_0, m_1]$,
\item $\{\delta(l_0), \delta(l_1) \}= [l_0, l_1]$,
\item $\{ \delta(l), m \}\{m, \delta(l)\}=1$,
\item $\{ m_0, m_1m_2 \}=\{m_0,m_1 \}\{m_0,m_2 \}\{ [m_2,m_0], m_1\}$,
\item $\{ m_0m_1, m_2 \}=\{m_0, m_1m_2m_1^{-1} \}\{m_1,m_2\}$.
\end{enumerate}
The category of reduced 2-crossed modules will be denoted as $R_2\X Mod$.
\end{defi} 
 
\begin{defi}\label{stablecrossdefi} \cite{Con84}
 A \textit{stable crossed module} is a group morphism $\delta: L \to M$ with a map $\{\ , \ \}: M\times M \to L$ satisfying:
\begin{enumerate}
\item $\delta \{m_0,m_1\}=[m_0,m_1]$,
\item $\{\delta(l_0), \delta(l_1)\}=[l_0,l_1]$,
\item $\{m_1, m_0 \}=\{m_0, m_1 \}^{-1}$,
\item $\{m_0m_1,m_2 \}=\{m_0m_1m_0^{-1} , m_0m_2m_0^{-1} \}\{m_0,m_2\}$.
\end{enumerate}
We will denote $St\X Mod$ for the category of stable crossed modules.
\end{defi}

The underlying morphism $\delta:L \to M$ of reduced 2-crossed module or of a stable crossed module is in fact a crossed module. Similar to the case of crossed modules/internal groupoids, these categories characterise simplicial groups via the normalization functor.

\begin{teo}\label{con} \cite{Con84}
\begin{enumerate}
\item The category $\M_{2 \geq}$ of simplicial groups with trivial Moore complex above degree 2 is equivalent to the category ${_2 \X Mod}$ of 2-crossed modules.
\item The category $\M_{1,2}$ of simplicial groups with trivial Moore complex except at degrees 1,2 is equivalent to the category $R_2 \X Mod$ of reduced 2-crossed modules.
\item The category $\M_{p,p+1}$ of simplicial groups with trivial Moore complex except at degrees $p,p+1$  for $p\geq 2$ is equivalent to the category $St_2 \X Mod$ of stable crossed modules.
\end{enumerate}
\end{teo}

The bottom torsion categories of $\M_{n \geq}$ can be characterized in a similar way as the torsion categories of internal groupoids.

\begin{teo}\label{teom2} For $n=2$, consider the lattice  of torsion theories $\mu(\M_{2 \geq})$:
\[ 0 \quad  \leq  \quad \mu_{ \geq 2}'   \quad \leq\ \quad  \mu_{1 \geq}' \quad  \leq \quad \mu'_{\geq 1} \quad \leq \quad  \mu_{0 \geq}' \quad  \leq \quad \M_{2 \geq}\]
For the bottom torsion categories we have the equivalences:
\begin{enumerate}
\item  the torsion category $\M_{ \geq 1}\cap \M_{2 \geq}= \M_{1,2}$ of $\mu_{ \geq 1}'$ is equivalent to the category $R_2\X Mod$ of reduced 2-crossed modules;
\item the torsion category $Ker(Cot_1)\cap \M_{2 \geq}$ of $\mu_{1 \geq }'$ is equivalent to the category $R_2\X Mod\cap CExt$ of reduced 2-crossed modules $\delta: L \to M$ with $\delta$ a central extension;
\item the torsion category $\M_{\geq 2}\cap \M_{2 \geq}$ of $\mu_{ \geq 2}'$ is equivalent to the category of 2-crossed modules of the form $A \to 0 \to 0$ and hence, it is also equivalent to the category $Ab$ of abelian groups.
\end{enumerate}
\end{teo}
\begin{proof}
1) It follows immediately from the fact that $\M_{ \geq 1}\cap \M_{2 \geq}$ is equivalent by definition to the category $\M_{1,2}$ of simplicial groups with trivial Moore complex except at degrees 1,2 and Theorem \ref{con}.

2) A simplicial group $X$ belongs to $Ker(Cot_1)\cap \M_{2 \geq}$ if and only if its Moore complex is trivial except $\delta_2:M_2 \to M_1$ which is, in addition, surjective (form Corollary \ref{cotnsurj}). This happens if and only if $\delta_2$ is a central extension, since a 2-reduced crossed module is in particular a crossed module.

3) The category $\M_{\geq 2}\cap \M_{2 \geq}$ is equivalent to the category of simplicial groups with trivial Moore complex except at degree 2, so it corresponds to a 2-crossed module of the form $L \to 0\to 0$. Since, in a 2-crossed module the morphism $\delta_2$ is a crossed module then $L$ must be an abelian group.
\end{proof}

\begin{teo}\label{teomn} For $n> 2$, in $\mu(\mathcal{M}_{n \geq})$ consider the  bottom torsion theories:
\[ 0 \quad  \leq  \quad \mu_{ \geq n}'   \quad \leq\ \quad  \mu_{n-1 \geq}' \quad  \leq \quad \mu'_{\geq n-1} \quad \dots  \]
Then for the torsion categories we have the equivalences:
\begin{enumerate}
\item  the torsion category $\M_{ \geq n-1}\cap \M_{n \geq}= \M_{n,n-1}$ of $\mu_{ \geq n-1}'$ is equivalent to the category $St \X Mod$ of stable crossed modules;
\item the torsion category $Ker(Cot_{n-1})\cap \M_{n \geq}$ of $\mu_{n-1 \geq }'$ is equivalent to the category $St\X Mod\cap CExt$ of stable crossed modules $\delta: L \to M$ with $\delta$ a central extension;
\item the torsion category $\M_{\geq n}\cap \M_{n \geq}$ of $\mu_{ \geq n}'$ is equivalent to the category $Ab$ of abelian groups.
\end{enumerate}
\end{teo}
\begin{proof} It follows from 3) of Theorem \ref{con}, the proof is similar to Theorem \ref{teom2}. 
\end{proof}

\begin{coro} A simplicial group $X$ belongs to the torsion category $Ker(Cot_{n-1})\cap \M_{n \geq}$ of $\M_{n \geq}$ if and only if its Moore complex $M$ is a central extension of groups.
\end{coro}

\begin{coro} The categories ${}_2\X Mod$ of 2-crossed modules, $R_2\X Mod$ of reduced 2-crossed modules and $St \X Mod$ of stable crossed modules are semi-abelian.
\end{coro}
\begin{proof} From Proposition \ref{semiab}, the categories $\M_{n \geq}$ are semi-abelian. In particular, for $n=2$ the category of 2-crossed modules is semi-abelian. Similarly, $R_2\X Mod$ and $St \X Mod$ are semi-abelian from Theorems \ref{teom2}, \ref{teomn} and Proposition \ref{semiab}.
\end{proof}

\subsection{The case of $\M_{\geq n}$}

A dual behaviour can be noticed when we work with the torsion categories $\M_{\geq n}$, which are  also semi-abelian (again form Proposition \ref{semiab}). When considering $\mu(\M_{\geq n})$, the restriction of $\mu(Grp)$ to $\M_{\geq n}$, we first obtain that the torsion theories above $\mu_{n\geq}$ are trivialized. And second,  that the upper torsion-free categories are equivalent to abelian groups and the different kinds of crossed modules.

\begin{prop} The lattice $\mu(\M_{\geq n})$ is given by:
\[0 \quad \leq \quad \dots \quad \leq \quad \mu_{n+1\geq}' \quad \leq \quad \mu_{\geq n+1}' \quad \leq \quad \mu_{n \geq}' \quad \leq \quad \M_{\geq n} \]
where $\mu_{i \geq}'$, $\mu_{\geq i}'$ are the restriction of the torsion theories of $\mu(Grp)$. 
\end{prop} 

\begin{teo}
For the category $\M_{\geq 1}$ we have 
\begin{enumerate}
\item the torsion-free category of $\mu_{1 \geq}'$ given by $\M_{1 \geq}\cap\M_{\geq 1}$ is equivalent to the category of abelian groups;
\item the torsion-free category of $\mu_{2 \geq}'$ given by $\M_{2 \geq}\cap\M_{\geq 1}$ is equivalent to the category $R_2 \X Mod$ of Reduced 2-crossed modules.
\end{enumerate}

For the category $\M_{\geq n}$ and $n \leq 2$ we have
\begin{enumerate}
\item the torsion-free category of $\mu_{n \geq}'$ given by $\M_{n \geq}\cap\M_{\geq n}$ is equivalent to the category of abelian groups;
\item the torsion-free category of $\mu_{n+1 \geq}'$ given by $\M_{n+1 \geq}\cap\M_{\geq n}$ is equivalent to the category $St \X Mod$ of stable crossed modules.
\end{enumerate}
\end{teo} 
\begin{proof} The proof is similar to the case of torsion categories in $\M_{n \geq}$.
\end{proof}

\section{Pretorsion theories in simplicial groups}

Pretorsion theories introduced in \cite{FacFin} present a generalization of a torsion theory for non-pointed categories. In brief words, we replace the zero object in the definition of a torsion theory in $\X$ for a class $\mathcal{Z}$ of \textit{trivial objects}.  We recall only the basic results and definitions useful for this work, we refer to \cite{FacFinGrn} for the basic aspects of pretorsion theories in categories.

Let $\X$ be a category and $\mathcal{Z}$ a class of objects of $\X$, which we will call trivial objects. A morphism $f: A \to B$ in $\X$ is called $\mathcal{Z}$-trivial if it is factorized by an object $Z$ in $\mathcal{Z}$.

Given a morphism $f:A \to B$ in $\X$ a morphism $k:K \to A$ is a $\mathcal{Z}$-\textit{prekernel} of $f$ if:
\begin{enumerate}
\item the composite $fk$ is $\mathcal{Z}$-trivial.
\item for any morphism $x: X \to A$ such that $ f \alpha$ is $\mathcal{Z}$-trivial then there is a unique $\lambda: X \to K$ such that $k \lambda= x$.
\end{enumerate}

The notion of a $\mathcal{Z}$-\textit{precokernel} is dually defined. Any $\mathcal{Z}$-prekernel is a monomorphism and any $\mathcal{Z}$-precokernel is an epimorphism. Given morphisms $f: A \to B$ and $g:B \to C$ the sequence
\[\begin{tikzcd} A\ar[r,"f"]&B\ar[r, "g"]&C\end{tikzcd}\]
is called a short $\mathcal{Z}$-preexact sequence if $f$ is a $\mathcal{Z}$-prekernel of $g$ and $g$ is a $\mathcal{Z}$-precokernel of $f$.

\begin{defi}\label{defprett} A pair $(\mathbb{T}, \mathbb{F})$ of full subcategories of a category $\X$ is a $\mathcal{Z}$-pretorsion theory in $\X$, with $\mathcal{Z}= \mathbb{T} \cap \mathbb{F}$, if:
\begin{enumerate}
\item any morphism $f: T \to F$ with $T$ in $\mathbb{T}$ and $F$ in $\mathbb{F}$ is $\mathcal{Z}$-trivial.
\item for any object $X$ in $\X$ there is a short $\mathcal{Z}$-preexact sequence:
\[\begin{tikzcd} T_X \ar[r, "\epsilon_X"] & X \ar[r, "\eta_X"] & F_X \end{tikzcd}\]
with $T_X$ in $\mathbb{T}$ and $F_X$ in $\mathbb{F}$.
\end{enumerate}
\end{defi}

Given a pretorsion theory $(\mathbb{T}, \mathbb{F})$ a pretorsion theory we have that $\mathbb{T}$ is a monocoreflective subcategory of $\X$ with counit $\epsilon_X$ in the Definition \ref{defprett}, dually, $\mathbb{F}$ is epireflective subcategory of $\X$ with unit $\eta_X$.

Now, we fix a semi-abelian category $\X$. We can define a pretorsion theory in $\X$ using torsion theories in $\X$. The next result was first noticed in the unpublished work \cite{Man15} in the context of crossed modules.

\begin{teo}\label{pretorsionteo} Let $\X$ be a semi-abelian category and torsion theories in $\X$, $(\T, \F)$ and $(\mathcal{S}, \mathcal{G})$ such that $ (\mathcal{S}, \mathcal{G}) \leq (\T, \F)$. For an object $X$ in $\X$ we have the associated short exact sequences
\[\begin{tikzcd} 0\ar[r] & T(X)\ar[r, "t_x"] & X\ar[r, "f_X"] & F(X)\ar[r] &0 \end{tikzcd}\]
and, respectively,
\[\begin{tikzcd} 0\ar[r] & S(X)\ar[r, "s_X"] & X\ar[r, "g_X"] & G(X)\ar[r] &0 \end{tikzcd} .\]
Then, the pair $(\mathbb{T}, \mathbb{F})= (\T, \mathcal{G})$ is a $\mathcal{Z}$-pretorsion theory in $\X$ and for an object $X$ we have the  associated  short $\mathcal{Z}$-preexact sequence:
\[\begin{tikzcd} T(X)\ar[r, "t_X"] & X\ar[r, "g_X"] & G(X) \, . \end{tikzcd} \]    
\end{teo}

\begin{proof}
To prove axiom 1 of a pretorsion theory consider a morphism  $\alpha: X \to Y$ with $X$ in $\T$ and $Y$ in $\mathcal{G}$. Consider the normal-epi/mono factorization of $\alpha$:
\[\begin{tikzcd} T\ar[rr, "\alpha"]\ar[rd, "e"']&&G \, .\\ &\alpha(T)\ar[ru, "m"']&
\end{tikzcd} \]
Since $\T$ is closed under quotients and $\mathcal{G}$ is closed under subobjects in $\X$ since they are a torsion and a torsion-free subcategory, respectively, of $\X$. It is clear that $\alpha(X)$ is in $\mathcal{Z}= \T \cap \mathcal{G}$.

To prove axiom 2 consider the sequence
\[\begin{tikzcd} T(X)\ar[r, "t_X"] & X\ar[r, "g_X"] & G(X) \end{tikzcd}. \]
We first notice that the composite factorizes through the object $T(X)/S(X)$:
\[\begin{tikzcd} S(X)\ar[rd]&&&&F(X) \, . \\ & T(X)\ar[r, "t_X"]\ar[rd]& X \ar[r, "g_X"]& G(X)\ar[ru, "q"]& \\&& T(X)/S(X)\ar[ru]&& \end{tikzcd} \]

The object $T(X)/S(X)$ is in $\T$ since is a quotient of $T(X)$, on the other hand it follows from Noether's third isomorphism that $T(X)/S(X)$ is the kernel of $q: G(X) \to F(X)$, so it belongs to $\mathcal{G}$. So, we have that the composite $g_Xt_X$ is $\mathcal{Z}$-trivial.

To prove that $t_x$ is a $\mathcal{Z}$-prekernel of $g_X$ consider a morphism $a: A \to X$ such that $g_xa$ is $\mathcal{Z}$-trivial. So, we have a commutative diagram
\[\begin{tikzcd} && F(X) \\ T(X)\ar[r, "t_X"] & X\ar[r, "g_X"]\ar[ru, "f_X"] & G(X)\ar[u, "q"'] \\ & A\ar[r,"b"]\ar[u,"a"] & Z\ar[u, "c"'] \end{tikzcd}\]
with $Z$ in $\mathcal{Z}$. Notice that the composite $f_x a=0$ since it factors through $qc$ and $Z$ is in $\T$. Since $t_X$ is the kernel of $f_x$ then we have a morphism $\lambda$ such that $a= t_x\lambda$ and then $t_x$ is a $\mathcal{Z}$-kernel of $g_X$. 

Dually, $g_x$ is a $\mathcal{Z}$-precokernel of $t_X$. Thus, we have a short $\mathcal{Z}$-preexact sequence.
\end{proof}  

\begin{rmk} Under the hypothesis of \ref{pretorsionteo}, we have seen that morphism $\alpha: X \to Y$ with $X$ in $\T$ and $Y$ in $\mathcal{G}$ is $\mathcal{Z}$-trivial, we  have noticed that the image of $\alpha$ belongs to $\mathcal{Z}$. However, beside the image factorization there are two other natural choices of factorizations of $\alpha$ to also prove this fact. Indeed, since torsion-free subcategories are epireflective $\alpha$ factors through the quotient $g_X: X \to G(X)$, and $G(X)$ belongs to $\mathcal{Z}$. Similarly, $\alpha$ also factors the monomorphism $t_Y: T(Y) \to Y$. Moreover, these factorizations satisfy that for any morphism $\alpha: X \to Y$ with epi/mono factorization $(e, m)$ there are unique morphisms $\beta$ and $\gamma$ such that the following diagram commutes:
\[\begin{tikzcd} & T(Y) \ar[rd, "t_Y"] & \\ X\ar[r, "e"]\ar[ru, "\alpha''"]\ar[rd, "g_X"'] & I\ar[r, "m"]\ar[u, "\gamma"] & Y \\ & G(X)\ar[ru, "\alpha'"']\ar[u, "\beta"] & \end{tikzcd} .\]
\end{rmk} 

\begin{coro}\label{trivialissemi} Consider the hypothesis as in \ref{pretorsionteo}. If the torsion theory $(\T, \F)$ is hereditary and the torsion theory $(\mathcal{S}, \mathcal{G})$ is cohereditary then the category of trivial objects $\mathcal{Z}= \T \cap \mathcal{G}$ is semi-abelian.
\end{coro}

\begin{proof} The pair $(\mathcal{S}, \mathcal{Z})$ is a torsion theory in $\T$. It follows from \ref{semiab} that $\T$ is semi-abelian. Then $\mathcal{Z}$, since is a cohereditary torsion-free subcategory of $\T$ is also semi-abelian (by \ref{semiab}).
\end{proof}

Now we turn our attention back to the lattice $\mu(Grp)$. Since this is an linearly ordered lattice, theorem \ref{pretorsionteo} produces a family of examples of pretorsion theories in $Simp(Grp)$. The trivial categories of these pretorsion theories show interesting properties.

\begin{teo}\label{pretorteo} Let $\mu(Grp)$ the lattice of torsion theories in $Simp(Grp)$ defined as in \ref{teolop}:
\begin{multline*}
0  \leq \dots \leq \; \mu_{ n+1\geq}  \;  \leq \;   \mu_{\geq n+1} \;  \leq \; \mu_{n\geq} \;   \leq \;  \mu_{\geq n} \;  \leq \dots \\  \dots \leq \; \mu_{\geq 2} \; \leq  \; \mu_{1\geq} \; \leq \;  \mu_{\geq 1} \;  \leq  \; \mu_{0\geq} \;   \leq Simp(Grp)\, . 
\end{multline*}
For the pretorsion theories defined by this lattice as in \ref{pretorsionteo} we have:
\begin{enumerate}
\item The trivial category $\mathcal{Z}$ given by the pair of torsion theories $\mu_{n \geq} \leq \mu_{ \geq n }$ is equivalent to the category $Ab$ of abelian groups. 
\item The trivial category $\mathcal{Z}$ given by the pair of torsion theories $\mu_{m \geq} \leq \mu_{ \geq n }$ is equivalent to the category of simplicial groups with truncated Moore complex above $m$ and  under $n$. Moreover, it is a semi-abelian category.
\item The trivial category $\mathcal{Z}$ given by the pair of torsion theories $\mu_{ \geq n+1} \leq \mu_{ n \geq }$ is equivalent to the category $Grp$ of groups.
\end{enumerate}
\end{teo}

\begin{proof}
\begin{enumerate}
\item We have that $\mathcal{Z}= \M_{n \geq}\cap \M_{\geq n}$, so the trivial objects have Moore complex of the form:
\begin{equation}\label{hi2}
\begin{tikzcd} \dots \ar[r]& 0\ar[r] & 0\ar[r] & A_n\ar[r] & 0\ar[r] & 0\ar[r] &\dots \end{tikzcd} 
\end{equation}
with $A_n$ an abelian group. From 3) of Theorem \ref{teomn} this category is equivalent to the category of abelian groups.
\item Clearly, $\mathcal{Z}= \M_{m \geq}\cap \M_{\geq n}$ correspond to the simplicial groups with truncated Moore complex above $m$ and under $n$. From Theorem \ref{trivialissemi}, since the torsion theory $\mu_{m \geq}$ is hereditary and  $\mu_{\geq n}$ is cohereditary then $\mathcal{Z}$ is semi-abelian.
\item We have that $\mathcal{Z}= \F_{tr_n} \cap Ker(Cot_n)$, from \ref{cotnsurj} it is clear that is equivalent to the category of simplicial groups with Moore complex of the form:
\begin{equation}\label{hi}
\begin{tikzcd} \dots\ar[r] & 0\ar[r] & 0\ar[r] & M_{n+1}\ar[r, "\cong"] & M_n\ar[r]& 0\ar[r] & 0 \dots\end{tikzcd}.
\end{equation}
We can identify three different cases, the case where the isomorphism occurs in degrees 0,1; the case where it occurs in degrees 1,2; and the general case for $n, n+1$ for $2<n$. On the other hand, let $G$ be a group and consider $id_G:G \to G$, it is straighforward from the Definition \ref{deficroos} that $id_G$ is a crossed module and that the structure of a crossed module is unique for a identity morphism. The category of these crossed modules is equivalent to the category is equivalent to $Ind(Grp)$ of indiscrete simplicial groups. Of course, we have $Ind(Grp) \cong Grp$, this proves the case 0,1. 

Moreover, given $G$, $id_G$ and setting $\{\ , \ \} =[\, ,\, ]$ as the commutator, $G$ is satisfies the equations of a reduced 2-crossed module and those of a stable crossed module (Definitions \ref{redcrossdefi} and \ref{stablecrossdefi}). Also, this structure of reduced 2-crossed module/stable crossed module is also necessarily unique. We can conclude that in any of the three different cases, $\mathcal{Z}$ is equivalent to $Grp$.

\end{enumerate}
\end{proof}

\begin{coro} Consider the pretorsion theories as in \ref{pretorteo}. Then, we have:
\begin{enumerate}
\item In the case of $\mu_{n \geq} \leq \mu_{ \geq n }$, any $\mathcal{Z}$-trivial simplicial group $X$ is isomorphic to the Eilenberg-Mac Lane simplicial group $K(\pi_n(x), n)$.
\item In the case of $\mu_{ \geq n+1} \leq \mu_{ n \geq }$, any  $\mathcal{Z}$-trivial simplicial group $X$ satisfy $\pi_k(X)=0$ for all $k$ (they are aspherical simplicial groups).
\end{enumerate}
\end{coro}

\begin{proof}
\begin{enumerate}
\item It is proved in \cite{Lop22b} (Lemma 7.4) that a simplicial group $X$ with a Moore complex as in Diagram \ref{hi2} in Theorem \ref{pretorteo} is unique up to isomorphism. In particular, it is isomorphic to $K(\pi_n(X), n)$ (see also the generalized Dold-Kan Theorem in \cite{CaCe91}).
\item It is well-known that the homotopy groups of a simplicial group $X$ can be calculated with the homology of its Moore complex $M$, $\pi_n(X)\cong H_n(M)$. On the other hand, it is clear that a complex as in the Diagram \ref{hi} in Theorem \ref{pretorteo} has trivial homology.
\end{enumerate}
\end{proof}

\begin{rmk} The Theorem \ref{pretorteo} asserts that from the lattice $\mu(Grp)$, two consecutive torsion theories yield a pretorsion theory with a semi-abelian trivial category $\mathcal{Z}$. Not any pair of torsion theories in $\mu(Grp)$ have this property. A remarkable example is when taking $\mu_{1 \geq} \leq \mu_{0 \geq}$ we have a pretorsion theory $(\mathbb{T}, \mathbb{F})$:
\begin{itemize}
\item where $\mathbb{T}$ is the category $Ker(Cot_0)$ of connected simplicial groups,
\item where $\mathbb{F}$ is the category $\M_{1 \geq}$ of groupoids,
\item where the trivial category $\mathcal{Z}$ is the category of trivial connected groupoids, which is not semi-abelian.
\end{itemize}
\end{rmk}

\section{Torsion theories in reduced crossed complexes} 

Introduced by M. K. Dakin \cite{Dak77}, a T-complex is a Kan simplicial object that admits a canonical filler for horns, for example internal groupoids have this property. Following the work of N. Ashley in \cite{Ash78} and the observations in \cite{CaCe91} a group T-complex (a T-complex in simplicial groups) can be defined as simplicial group $X$ with $M_n \cap D_n=0$ where $M$ is the Moore complex of $X$ and $D$ is the graded subgroup of $X$ generated by the degenerated elements of $X$.

\begin{defi}\cite{Ash78}
A \textit{reduced crossed complex} $M$, or a crossed complex in groups, is a proper chain complex
\[M= \begin{tikzcd} \dots \ar[r]& M_n \ar[r, "\delta_n"]& M_{n-1}\ar[r]& \dots\ar[r] & M_2\ar[r, "\delta_2"] & M_1\ar[r, "\delta_1"] & M_0\end{tikzcd} \]
where  
\begin{enumerate}
\item $M_n$ is abelian for $n \geq 2$;
\item $M_0$ acts on $M_n$ for $n \geq 1$ and the restriction to $\delta_1(M_1)$ acts trivially on $M_n$ for $n \geq 2$;
\item $\delta_n$ preserves the action of $M_0$ and $\delta_1:M_1 \to M_0$ is a crossed module.
\end{enumerate}
A morphism of reduced crossed complexes is a chain complex morphism that preserves all actions. We will write $Crs(Grp)$ for the category of reduced crossed complexes.
\end{defi}

It is proved in \cite{Ash78} that the category of group T-complexes is equivalent to the category of reduced crossed complexes via the Moore normalization. And in \cite{EhPo97} it is shown that the category of reduced crossed complexes is an epi-reflective subcategory of simplicial groups.

In \cite{CaCe91}, the notion of an hypercrossed module is introduced, a hypercrossed module $M$ is a chain complex in groups with group actions 

\[\Phi_i^j : \begin{tikzcd} M_i \ar[r] & Aut(M_j) \end{tikzcd}\] 

and binary operations

\[\Gamma^k_{i,j}: \begin{tikzcd} M_i \times M_j \ar[r] & M_n \end{tikzcd}\]
where the indexes $i,j,k$ are determined by the order introduced by Conduch\'e \cite{Con84} in the set $S(n)$ of surjective maps with domain $[n]$ of the  simplicial category $\Delta$ and this data must satisfy some equations. Crossed modules and 2-crossed modules are hypercrossed modules, indeed, the Peiffer lifting is an example of these binary operations. Furthermore, Ashley's reduced crossed modules are hypercrossed modules where all binary operations are trivial, $\Gamma^k_{i,j}=0$. Also, in \cite{CaCe91} is established an equivalence between simplicial groups and hypercrossed modules, a generalized Dold-Kan Theorem, this expands the equivalences between crossed modules/internal groupoids, 2-crossed modules/simplicial groups with trivial Moore complex above 2 and, finally, reduced crossed complexes/$T$-group complexes.

\begin{coro} The category of Dakin's group $T$-complexes is semi-abelian. Hence, the category $Crs(Grp)$ of reduced crossed complexes is also semi-abelian.
\end{coro}
\begin{proof} From \cite{EhPo97}, we have that the category of group $T$-complexes is a normal epi-reflective subcategory of $Simp(Grp)$. In fact, it is a Birkhoff subcategory so it is semi-abelian. Indeed, we just need to prove that it is closed under regular epimorphism in $Simp(Grp)$. So, let $X$ be a group $T$-complex and a regular epimorphism $f: X \to Y$ in $Simp(Grp)$. The equivalence between crossed complexes and $T$-complexes is given by the Moore normalization, so if $M^X$
and $M^Y$ are the Moore complexes of $X$ and $Y$, it remains to see that since $M^X$ is a reduced crossed complex then so is $M^Y$. The Moore normalization preserves regular epimorphisms, so $N(f):M^X \to M^Y$ is a surjective map component-wise. A morphism of hypercrossed modules (and, thus, of reduced crossed complexes) is a chain complex morphism compatible with the actions $\Phi_i^j$ and the operations $\Gamma^k_{i,j}$. So, we have a commutative diagram
\[\begin{tikzcd}[column sep= large] M^X_i \times M^X_j\ar[d, "N(f) \times N(f)"', two heads]  \ar[r, "\Gamma^{X,k}_{i,j}"] & M^X_k  \ar[d, "N(f)", two heads]\\ M^Y_i \times M^Y_j  \ar[r, "\Gamma^{Y,k}_{i,j}"] & M^Y_k \, . \end{tikzcd} \]

 Then, if $\Gamma^{X,k}_{i,j}=0$, since $M_X$ is a reduced crossed complex, then we have $\Gamma^{Y,k}_{i,j}=0$. Finally, $M^Y$ is a reduced crossed complex and, equivalently, $Y$ is $T$-group complex.
\end{proof}

Our interest in reduced crossed complexes lies in their similar behaviour to chain complexes, thus torsion theories of simplicial groups can be easily studied when restricted to subcategory of $Crs(Grp)$. To this end, we recall some properties of reduced crossed complexes, we refer the reader to \cite{BrHiSi10} for the details and proofs.

\begin{strc}\label{adj}
An $n$-reduced crossed complex is an $n$-truncated chain complex $M$:
 \[M= \begin{tikzcd} M_n \ar[r, "\delta_n"]& M_{n-1}\ar[r]& \dots\ar[r] & M_2\ar[r, "\delta_2"] & M_1\ar[r, "\delta_1"] & M_0\end{tikzcd} \]
satisfying all the axioms of a reduced crossed complex (those that make sense), thus we have a category $Crs(Grp)_{n \geq}$ of $n$-truncated crossed complexes. Clearly, a 1-truncated reduced crossed complex is nothing but a crossed module, thus $ Crs(Grp)_{1 \geq} = \X Mod$.

Moreover, we have the functors for all $n\in \mathbb{N}$:
\begin{itemize}
\item the truncation functor $\mathbf{tr}_n: Crs(Grp) \to Crs(Grp)_{n \geq}$
\[\mathbf{tr}_n(M)=\begin{tikzcd} M_n \ar[r] & M_{n-1}\ar[r] & M_{n-2} \ar[r] & \dots  \end{tikzcd};\]
\item the skeleton functor (or natural embedding) $\mathbf{sk}_n: Crs(Grp)_{n \geq} \to Crs(Grp)$
\[\mathbf{sk}_n(M)= \begin{tikzcd} \dots\ar[r] & 0\ar[r]&0\ar[r] &M_n\ar[r] &M_{n-1}\ar[r] & \dots \end{tikzcd}; \]
\item the coskeleton functor $\mathbf{cosk}_n: Crs(Grp)_{n \geq} \to Crs(Grp)$
\[\mathbf{cosk}_n(M)= \begin{tikzcd} \dots\ar[r] & 0\ar[r]&ker(\delta_n)\ar[r] &M_n\ar[r] &M_{n-1}\ar[r] & \dots \end{tikzcd}; \]
\item the cotruncation $\mathbf{cot}_n: Crs(Grp) \to Crs(Grp)_{n \geq}$
\[\mathbf{cot}_n(M)=\begin{tikzcd} M_n/\delta_{n+1}(M_{n+1}) \ar[r] & M_{n-1}\ar[r] & M_{n-2} \ar[r] & \dots  \end{tikzcd}.\]
\end{itemize}
We will write $\mathbf{Sk}_n=\mathbf{sk}_n\mathbf{tr}_n$, $\mathbf{Cosk}_n=\mathbf{cosk}_n\mathbf{tr_n}$ and $\mathbf{Cot}_n=\mathbf{sk}_n\mathbf{cot}_n$. These functors give a string of adjunctions
\[\mathbf{cot}_n \dashv \mathbf{sk}_n \dashv \mathbf{tr}_n \dashv \mathbf{cosk}_n:  
\begin{tikzcd}[row sep=large] 
Crs(Grp)
\ar[d, bend left, shift right=.5] \ar[d, bend right, shift right=4] \ar[d, phantom, "\dashv"] \ar[d,phantom,"\dashv", shift right=5.5 ] \ar[d,phantom, "\dashv", shift left=5.5 ] \\
Crs(Grp)_{n \geq }  
\ar[u, bend left, shift right=.5] \ar[u, bend right, shift right=4]
\end{tikzcd} \, .\]

It is worth mentioning that the adjunctions of simplicial groups $sk_n \dashv tr_n \dashv cosk_n$ restricted to crossed complexes correspond (under Moore normalization) to $\mathbf{sk}_n \dashv \mathbf{tr}_n \dashv \mathbf{cosk}_n$ and Porter's cotruncation $Cot_n$ corresponds to $\mathbf{Cot}_n$. In addition, unlike the general case of simplicial groups, the adjunction $\mathbf{cot}_n\dashv \mathbf{sk}_n$ holds for reduced crossed complexes.   

On the other hand, it will be useful to write $Crs(Grp)_{\geq n}$ for the subcategory of reduced crossed complexes with $M_i=0$ for $n>i$. For all $n\geq 1$, $Crs(Grp)_{\geq n}$ is equivalent to the category of $chn(Ab)_{\geq n}$ chain complexes of abelian groups. Moreover, we can easily define the dual functors:
\begin{itemize}
\item $\mathbf{tr}'_n:Crs(Grp) \to Crs(Grp)_{\geq n}$;
\item $\mathbf{sk}'_n:Crs(Grp)_{\geq n} \to Crs(Grp)$;
\item $\mathbf{cot}'_n:Crs(Grp) \to Crs(Grp)_{\geq n}$.
\end{itemize}
However, only the adjunction $\mathbf{sk}'_n \dashv \mathbf{cot}'_n$ holds.
\end{strc}

\begin{prop}\label{ttcrs} Let $\mu(Crs(Grp))$ be the lattice given by the restriction of torsion theories in $\mu(Grp)$ to $Crs(Grp)$:
\[ \mu(Crs(Grp)) = \quad \dots \; \leq \quad \mu_{\geq 2}' \quad \leq \quad \mu_{1 \geq}'  \quad \leq \quad \mu_{\geq 1}' \quad \leq \quad \mu_{0 \geq}' \quad \leq \quad Crs(Grp)\, . \]
Then, the torsion theories $\mu'_{n \geq}$ and $\mu'_{ \geq n}$ can be expressed with the functors $\mathbf{cot}_n \dashv \mathbf{sk}_n \dashv \mathbf{tr}_n \dashv \mathbf{cosk}_n$:
\[\mu'_{n \geq}=(Ker(\mathbf{cot}_n) ,Crs(Grp)_{n \geq}), \quad \mu'_{\geq n}=(Crs(Grp)_{\geq n} , \F_{\mathbf{tr}_n} )\, .\]
\end{prop}
\begin{proof} It follows from the fact that the equivalence between group $T$-complexes and reduced crossed complexes is given by the Moore normalization, and also from the short exact sequences in Theorem \ref{teolop} and the previous observations \ref{adj}. 
\end{proof}

\begin{teo} For $n>1$, let $\mu(Crs(Grp)_{n \geq})$ be the lattice given by restriction of $\mu(Grp)$ to $Crs(Grp)_{n \geq}$:
\[\mu(Crs(Grp)_{n \geq})=  0 \leq \quad \mu_{\geq n}' \quad  \leq \quad  \mu_{n-1 \geq}' \quad \leq \; \dots  \; \leq \quad \mu_{\geq 1}' \quad \leq \quad \mu_{0 \geq}' \quad \leq \quad Crs(Grp)_{n \geq}\, . \]
Then for the torsion categories we have the equivalences:
\begin{enumerate}
\item $Ker(\mathbf{cot}_{n-1})\cap Crs(Grp)_{n \geq}$, the torsion category of $\mu'_{n-1 \geq}$, is equivalent to the category $CExt(Ab)$ of central extension in abelian groups, i.e., surjective morphisms.
\item $Crs(Grp)_{\geq n} \cap Crs(Grp)_{n \geq}$, the torsion category of $\mu'_{\geq n}$, is equivalent to the category $Ab$ of abelian groups.
\end{enumerate}
\end{teo}
\begin{proof} The proof is similar to Theorem \ref{teom2}.
\end{proof}

\section{Torsion torsion-free theories}

\begin{defi} A \textit{torsion torsion-free theory} in $\X$, or TTF theory, is a triplet $(\C, \T, \F)$ of full subcategories of $\X$ such that $(\C, \T)$ and $(\T, \F)$ are torsion theories in $\X$. A subcategory $\T$  of $\X$ is called a torsion torsion-free category (or a TTF category) if there are subcategories $\C$ and $\F$ such that $(\C, \T, \F)$ is a TTF theory.
\end{defi}

TTF theories were introduced in \cite{jans65} with applications mainly to categories of modules over rings, in particular they are used to study when any object of a category is a joint of torsion subobjects for different torsion theories. TTF theories have been studied in different non-abelian settings, for example in triangulated categories in \cite{BeRe07}. In our context of simplicial groups, we present examples of TTF theories in a weak sense.

\subsection{TTF theories in chain complexes}

Let $\X$ be a semi-abelian category and $chn(\X)$ and $pch(\X)$ the categories of chain complexes and proper chain complexes, respectively. We will write $chn(\X)_{n \geq}$ and $chn(\X)_{\geq n}$ for the category of chain complexes bounded above/below $n$.

We have the adjunctions:
\[ \mathbf{cot}_n \dashv \mathbf{sk}_n \dashv \mathbf{tr}_n \dashv \mathbf{cosk}_n:  
\begin{tikzcd}[row sep=large] 
\chn 
\ar[d, bend left, shift right=.5] \ar[d, bend right, shift right=4] \ar[d, phantom, "\dashv"] \ar[d,phantom,"\dashv", shift right=5.5 ] \ar[d,phantom, "\dashv", shift left=5.5 ] \\
\chn_{n \geq }
\ar[u, bend left, shift right=.5] \ar[u, bend right, shift right=4]
\end{tikzcd} ;\] 
as well as their duals:
\[ \mathbf{cosk}'_n \dashv \mathbf{tr}'_n \dashv \mathbf{sk}'_n \dashv \mathbf{cot}'_n: 
\begin{tikzcd}[row sep=large]
\chn_{ \geq n} 
\ar[d, bend left, shift right=.5] \ar[d, bend right, shift right=4] \ar[d, phantom, "\dashv"] \ar[d,phantom,"\dashv", shift right=5.5 ] \ar[d,phantom, "\dashv", shift left=5.5 ] \\
\chn 
\ar[u, bend left, shift right=.5] \ar[u, bend right, shift right=4]
\end{tikzcd} .
\]
defined similarly as in \ref{adj}. These adjunctions can be restricted to adjunctions between $pch(\X)$ and $pch(\X)_{n \geq}$ (or $pch(\X)_{\geq n}$ respectively).

Similar to simplicial groups, torsion theories are given by the cotruncation functors $\mathbf{cot}$ and $\mathbf{cot}'$ as follows.

\begin{teo}\label{ttchcompl} \cite{Lop22b} Let $\X$ be a semi-abelian category, then we have: 
\begin{enumerate}
\item in $chn(\X)$ we have a torsion theory $(chn(\X)_{\geq n}, \F_{\mathbf{tr}_{n-1}})$:
 \[\begin{tikzcd}[column sep= large, row sep=huge] \chn_{\geq n}\ar[r, bend left, "\mathbf{sk}_n'" above]\ar[r, phantom, "\perp"]  & \chn \ar[l, bend left, "\mathbf{cot}_n'" below]\ar[r, bend left]\ar[r,phantom, "\perp"] &  \F_{\mathbf{tr_{n-1}}} \ar[l, bend left] \end{tikzcd} ; \]
\item the previous torsion theory is restricted to  a torsion theory in $pch(\X)$ as $(pch(\X)_{\geq n}, \mathcal{MN}_n)$:
\[\begin{tikzcd}[column sep= large, row sep=huge]  \pch_{ \geq n}\ar[r, bend left, "\mathbf{sk}_n'" above]\ar[r, phantom, "\perp"]  & \pch \ar[l, bend left, "\mathbf{cot}_n'" below]\ar[r, bend left]\ar[r,phantom, "\perp"] &  \mathcal{MN}_n  \ar[l, bend left] \end{tikzcd} ; \]
\item the category $chn(\X)_{n \geq}$ is a normal epi-reflective subcategory of $chn(\X)$, but it is not a torsion-free subcategory ;
\item in $pch(\X)$ we have a torsion theory $(\mathcal{EP}_n, pch(\X)_{n-1 \geq})$:
\[\begin{tikzcd}[column sep= large, row sep=huge]  \mathcal{EP}_n\ar[r, bend left]\ar[r, phantom, "\perp"] & \pch \ar[l, bend left ]\ar[r, bend left, "\mathbf{cot}_{n-1}" above]\ar[r,phantom, "\perp"] &  \pch_{n-1\geq } \ar[l, bend left, "\mathbf{sk}_{n-1}" below] \end{tikzcd} ;\]
\end{enumerate}
where 
\begin{itemize}
\item $\F_{\mathbf{tr}_{n-1}}$ is the subcategory of chain complexes such that their component of the unit of $\mathbf{tr}_{n-1}\dashv \mathbf{cosk}_{n-1}$ is monic. 
\item $\mathcal{MN}_n$ is the subcategory of  proper chain complexes such that $M_i=0$ for $i>n$ and the differential $\delta_n$ is monic. 
\item $\mathcal{EP}_n$ is the subcategory of  proper chain complexes such that $M_i=0$ for $n-1>i$ and the differential $\delta_n$ is epic.
\end{itemize}
For a chain complex $M$, the short exact sequences of the torsion theories in 2) and 3) are given as in diagrams (3) and (4) in Theorem \ref{teolop}, respectively.
\end{teo}

In addition to this torsion theories, we can add the following.

\begin{teo} Let $\X$ be a semi-abelian category. For each $n \in \Z$ the pair $(\chn_{n-1 \geq}, \chn_{\geq n})$ is a hereditary cohereditary torsion theory in $\chn$. Moreover, the reflector and coreflector are given by $\mathbf{sk}_{n-1} \dashv \mathbf{tr}_{n-1}$ and $\mathbf{tr}'_n \dashv \mathbf{sk}'_n$:
\[\begin{tikzcd}[column sep= large, row sep=huge] 
\chn_{n-1 \geq}\ar[r, bend left, "\mathbf{sk}_{n-1}" above]\ar[r, phantom, "\perp"]  
& \chn \ar[l, bend left, "\mathbf{tr}_{n-1}" below]\ar[r, bend left, "\mathbf{tr}'_n"]\ar[r,phantom, "\perp"] 
&  \chn_{\geq n} \ar[l, bend left, "\mathbf{sk}'_n"] \end{tikzcd} . \]
\end{teo}
\begin{proof} For $X$ in $\chn_{n-1 \geq}$ and $Y$ in $\chn_{\geq n}$ it is clear that a morphism $\mathbf{sk}_{n-1}(X) \to \mathbf{sk}'_n(Y)$ must be trivial:
\[\begin{tikzcd}[column sep=scriptsize, row sep=scriptsize]  
\mathbf{sk}_{n-1}(X)\ar[d]= & \dots\ar[r] &  0\ar[r] \ar[d]& 0 \ar[r] \ar[d]& X_{n-1}\ar[r]\ar[d] & X_{n-2} \ar[d]
\\
\mathbf{sk}'_n(Y)= & \dots \ar[r] & Y_{n+1}\ar[r] & Y_n \ar[r] &  0 \ar[r]& 0 \, .
\end{tikzcd}  \]
Since limits and colimits are computed component-wise in $\chn$, the short exact sequence of the torsion theory for a chain complex $X$ in $\chn$ is given by:
\[ \begin{tikzcd}[column sep=scriptsize, row sep=scriptsize]
 \dots \ar[r] &  0\ar[r] \ar[d]& 0 \ar[r] \ar[d]& X_{n-1}\ar[r]\ar[d] & X_{n-2} \ar[d]\ar[r] & \dots
\\
 \dots \ar[r] & X_{n+1}\ar[r] \ar[d] & X_n \ar[r]\ar[d] &  X_{n-1} \ar[r] \ar[d]& X_{n-2}\ar[r]\ar[d] & \dots
\\
 \dots \ar[r] & X_{n+1}\ar[r] & X_n \ar[r] &  0 \ar[r] & 0 \ar[r]& \dots
\end{tikzcd}  \]
\end{proof}

\begin{coro}\label{ttfpch} Let $\X$ be semi-abelian category. For each $n \in \Z$ the triplet of full subcategories
 \[(\chn_{n-1 \geq}, \chn_{\geq n}, \F_{\mathbf{tr}_{n-1}})\]   is a TTF theory in $\chn$. Moreover, by restriction  this determines the TTF theory in $\pch$
 \[(\pch_{n-1 \geq}, \pch_{\geq n}, \mathcal{MN}_n).\]
Similarly, the triplet of subcategories in $\chn$ 
\[(Ker(\mathbf{cot}_{n-1 \geq}),\chn_{n-1 \geq} , \chn_{\geq n} )\]
determines the TTF theories in $\pch$
 \[  (\mathcal{EP}_n, \pch_{n-1 \geq}, \pch_{ \geq n}).\]
\end{coro}

\subsection{Weak TTF theories in chain complexes with operations}

Through this work we have mentioned the similarities between torsion theories in simplicial groups and in chain complexes, this happens since they are both defined with similar set of adjunctions ($sk_n\dashv tr_n \dashv cosk_n$ and  $\mathbf{cot}_n \dashv \mathbf{sk}_n \dashv \mathbf{tr}_n \dashv \mathbf{cosk}_n$). However, the TTF theories in chain complexes cannot be easily adapted to the simplicial case, not even in the initial example of internal groupoids. For instance, in $Grpd(Grp)$ the subcategory of discrete groupoids $Dis(Grp) \cong \mathcal{M}_{0 \geq}$ is a torsion-free subcategory (with reflector $\pi_0\dashv Dis$) and also mono-coreflective (with the adjunction $Dis \dashv tr_0$) but not normal mono-coreflective, so it is not a torsion subcategory. In general, the subcategories $\mathcal{M}_{ n \geq}$ are only torsion-free. 

Two different kinds of torsion theories are introduced.

\begin{defi}
 For a class of objects $\mathcal{E}$ of $\X$, a pair $(\T, \F)$ of full subcategories of $\X$ will be called a $\mathcal{E}$-torsion theory or a torsion theory relative to the class $\mathcal{E}$ if:
\begin{itemize}
\item[TT1] for all $X \in \T$ and $Y \in \F$, every morphism $f:X \to Y$ is zero;  
\item[TT2'] for every object $X \in \mathcal{E}$ exists a short exact sequence
\[\begin{tikzcd} 0\ar[r] & T_X\ar[r, "t_X"] & X\ar[r, "f_X"] & F_X\ar[r] &0 \end{tikzcd}\]
with $T_X \in \T$ and $F_X \in \F$.
\end{itemize}
\end{defi}

As a first example of an $\mathcal{E}$-torsion theory we have the pair $(Ker(\mathbf{cot}_n), \chn_{n \geq})$ in $chn(\X)$ as in Theorem \ref{ttchcompl}.

\begin{lemma} In $\chn$ the category of chain complexes the pair 
\[(Ker(\mathbf{cot}_n), \chn_{n \geq})\] and $\mathcal{E}$ the class of proper chain complexes $\pch$ is a $\mathcal{E}$-torsion theory in $\chn$.
\end{lemma}
\begin{proof} The objects in $Ker(\mathbf{cot}_n)$ are the chain complexes $X$ such that $X_i=0$ for $n> i$ and the differential $\delta_{n+1}$ has a trivial cokernel. Thus, to verify TT1 it suffices to notice that given a commutative diagram
\[\begin{tikzcd} X_{n+1}\ar[r, "\delta_{n+1}"]\ar[d, "f_{n+1}"] & X_n\ar[d, "f_n"] \\  0\ar[r] & Y_n \end{tikzcd}\]
with $\delta_{n+1}$ a morphism with trivial cokernel then the morphism $f$ must be trivial. 
TT2' holds since it has been established that the restriction of the pair $(Ker(\mathbf{cot}_n), \chn_{n \geq})$ to proper chain gives a torsion theory $(\mathcal{EP}_n, \pch_{n \geq})$ in $\pch$. In addition, for a proper chain complex $M$ the short exact sequence is given by diagram  (3) in Theorem \ref{teolop}.
\end{proof}

Any normal epireflective subcategory yields a $\mathcal{E}$-torsion theory. Let $F \dashv I: \X \to \A$ be a normal epireflective subcategory of $\X$ with unit $\eta$. Following \cite{BouGrn06}, we consider two subcategories of $\X$ :
\[\T_\F=\{T \mid T \cong ker(\eta_X) \; \mbox{for some} \; X\} \]
and
\[Ker(F)= \{X \mid F(X)=0 \}.\] 
Clearly, we have $Ker(F) \subseteq \T_\F$. The pair $(\T_\F , \F)$ satisfies axiom TT2 of a torsion theory while the pair $(Ker(F), \F)$ satisfy axiom TT1. Indeed, if $Ker(F)= \T_\F$ we have a torsion theory. Thus, in the relative case we have the following:

\begin{lemma}   Let $F \dashv I: \X \to \A$ be a normal epireflective subcategory of $\X$ with unit $\eta$. If $\mathcal{E}=\{X \mid F(ker(\eta_X))=0 \}$ then the pair $(Ker(F), \F)$ is a $\mathcal{E}$-torsion theory.
\end{lemma} 

\begin{exmp} The category of $Ab$ of abelian groups is a normal epireflective subcategory of $Grp$ (in fact, it is a Birkhoff subcategory) where the reflector is the abelianization functor $ab(G)=G/G'$ where $G'$ is the commutator subgroup. Hence, if $Perf$ is the category of perfect groups, groups $G$ such that $G'=G$, then the pair $(Perf, Ab)$ is a $\mathcal{E}$-torsion theory with respect the class $\mathcal{E}$ of groups such that $(G')'=G'$ or, equivalently, groups with a perfect commutator. 
\end{exmp}

We introduce our main definition.

\begin{defi}
Let $(\T, \F)$ be a torsion theory in $\X$, if $\F$ is a mono-coreflective category of $\X$ we will call $(\T, \F)$ a \textit{CTF theory}. This means, that the embedding  $I$ of $\F$ into $\X$ has both a right and left adjoint, $F\dashv I \dashv C$:
\[\begin{tikzcd}[column sep=large] \T \ar[r, bend left, "J"]\ar[r, phantom, "\perp"] & \X \ar[r, bend left, "F"]\ar[l, bend left,"T"]\ar[r, phantom, "\perp"]\ar[r, bend right, "C"', shift right=6]\ar[r, phantom, "\perp", shift right=6] & \F \ar[l, bend left, "I", near start] \end{tikzcd} .\]
\end{defi}

 Clearly, in a TTF theory $(\mathcal{C}, \T, \F)$ the pair $(\mathcal{C}, \T)$ is a CTF-theory. In \cite{CDT06}, it is proved that a normal mono-coreflective subcategory closed under extensions is in fact a torsion category, so a CTF theory with $\F$ a normal-monocoreflective  subcategory is a TTF theory.
 
The subcategory of discrete crossed modules $Dis$ in $\X Mod$ behaves almost as a torsion torsion-free subcategory, it presents an example of a CTF theory as well as a relative $\mathcal{E}$-torsion theory. To this end, recall that a monomorphism in $\X Mod$ is given by an injective crossed module morphism $f=(f_1,f_0)$:
\[\begin{tikzcd} A\ar[r]\ar[d, "f_1"] & B\ar[d, "f_0"] \\ X\ar[r] & Y\, . \end{tikzcd}\]
In addition, $f$ is a normal subcrossed module if and only if $A$, $B$ are normal subgroups of $X$, $Y$, and the conditions  ${}^{y}(a)\in A$ and ${}^b(x)x^{-1}\in X$ hold for all $a \in A$, $b \in B$, $y\in Y$ and $x\in X$.

\begin{prop}\label{ctfxmod} In $\X Mod$ consider the triplet of subcategories
\[(CExt, Dis, Ab )\]
then:
\begin{enumerate}
\item the pair $(CExt ,Dis)$ is a CTF theory in $\X Mod$;
\item the pair $(Dis, Ab)$ is an $\mathcal{E}$-torsion theory where $\mathcal{E}$ is the class of crossed modules $\delta:A \to B$ where the action $B \to Aut(A)$ is trivial.
\end{enumerate}
\end{prop}
\begin{proof} 1) The discrete functor $D$ has right adjoint $(\, )_0$ where the component of the counit for a crossed module $\delta: A \to B$ is given by horizontal arrows in the diagram:
\begin{equation}\label{counitdis0}
\begin{tikzcd} 0\ar[r,"0"]\ar[d] & A\ar[d,"\delta"] \\ B\ar[r, "1"] & B \end{tikzcd} 
\end{equation}
which is a monomorphism since the pair $(0,1)$ are injective morphisms. 

2) It is clear that the pair $(Dis, Ab)$ satisfies TT1 since in a commutative diagram
\[\begin{tikzcd} 0\ar[r, "f_1"]\ar[d] & A\ar[d] \\ G\ar[r, "f_0"]& 0 \end{tikzcd}\]
the morphism $f=(f_1,f_0)$ is zero. For TT2', recall that the unit Diagram (\ref{counitdis0}) is a normal monomorphism in $\X Mod$ if and only if $^b(a)a^{-1}=0$, i.e., the action of $B$ over $A$ is trivial. From the Peiffer identity $^{\delta(a)}(a')=aa'a^{-1}$, a crossed module with trivial action also has $A$ as an abelian group then we have the short exact sequence in $\X Mod$:
\[\begin{tikzcd} 0\ar[r]&0\ar[r]\ar[d]&A\ar[r]\ar[d, "\delta"]&A\ar[r]\ar[d]&0 \\ 0\ar[r]&B\ar[r]&B\ar[r] &0\ar[r]&0 \, . \end{tikzcd} \]   
\end{proof}

\begin{rmk} The pair of subcategories $Dis$ and $Ab$ of $\X Mod$ present another example of an $\mathcal{E}$-torsion theory in $\X Mod$. This time as $(Ab, Dis)$ and $\mathcal{E}$ as the subcategory $Mod$ of modules of groups. For a module we mean a pair $(A, G)$ such that $G$ is a group and $A$ is an abelian groups with a group action $G \to Aut(A)$. Then a module $(A,G)$ is a crossed modules as $\delta=0:A\to G$. In fact, the subcategory $Mod$ is a Birkhoff subcategory of $\X Mod$. The associated short exact sequence of the $\mathcal{E}
$-torsion theory for a module $(A, G)$ is
\[\begin{tikzcd} 0\ar[r]&A\ar[r]\ar[d]&A\ar[r]\ar[d, "\delta=0"]&0\ar[r]\ar[d]&0 \\ 0\ar[r]&0\ar[r]&G\ar[r] &G\ar[r]&0 \, . \end{tikzcd} \]   
\end{rmk}

The category of reduced crossed complexes present a similar example as Proposition \ref{ctfxmod}.

\begin{strc} In $Crs(Grp)$ for each $n\geq 0$  consider the full subcategory $Crs(Grp)_{\geq n}$ of reduced crossed complexes $M$ who are trivial in degrees below $n$:
\[M= \begin{tikzcd}\dots\ar[r] & M_{n+1} \ar[r]  & M_n \ar[r] & 0\ar[r] & 0\ar[r] & \dots\, . \end{tikzcd}  \]
For all $n>0$ the category $Crs(Grp)_{\geq n}$ is equivalent to the category $ch(Ab)_{\geq n}\cong ch(Ab)$ of chain complexes in abelian groups.

Thus, for $n \geq 2$ we have a functor $\mathbf{tr'}_n: Crs(Grp) \to Crs(Grp)_{\geq n}$ defined for a crossed complex $M$ by
\[\mathbf{tr'}_n (M)= \begin{tikzcd}\dots\ar[r] & M_{n+1} \ar[r]  & M_n \ar[r] & 0\ar[r] & \dots  \end{tikzcd}. \]
The natural  chain complex morphism $f:M \to \mathbf{tr'}_n(M)$:
\[ \begin{tikzcd}\dots\ar[r] & M_{n+1}\ar[d, "1"] \ar[r]  & M_n \ar[r]\ar[d, "1"] & M_{n-1}\ar[r]\ar[d, "0"] & \dots \ar[r] & M_1\ar[r]\ar[d, "0"] & M_0\ar[d, "0"] \\ \dots\ar[r] & M_{n+1} \ar[r]  & M_n \ar[r] & 0\ar[r] & \dots\ar[r] & 0\ar[r] & 0 \end{tikzcd} \]
is a morphism in $Crs(Grp)$ if and only if all the actions $M_0\to Aut(M_i)$ are trivial for $i \geq n$. Indeed, $M$ should satisfy $ ^{m_0}m_n= f_n(^{m_0}m_n)= ^{f_0(m_0)}f_n(m_n)=m_n$ for all $m_0 \in M_0$ and $m_n \in M_n$. In particular, this condition holds if $\delta_1: M_1 \to M_0$ is a central extension, since in a crossed complex the restrictions of the actions $\delta_1(M_1) \to Aut(M_0)$ are trivial.
\end{strc} 

\begin{prop}\label{teocrs} For $n \geq 2$, consider the triplet of subcategories:
\[ (Ker(\mathbf{cot}_{n-1}) , Crs(Grp)_{n-1 \geq}, Crs(Grp)_{ \geq n} )\, . \]
in $Crs(Grp)$. Then:
\begin{enumerate}
\item the pair $(Ker(\mathbf{cot}_{n-1}), Crs(Grp)_{n-1 \geq})$ is CTF theory, i.e., the subcategory  $Crs(Grp)_{n-1 \geq}$ is mono-coreflective;
\item the pair  $(Crs(Grp)_{n-1 \geq}, Crs(Grp)_{ \geq n})$ is an $\mathcal{E}$-torsion theory where $\mathcal{E}$ is the class of crossed complexes $M$ with all actions $M_0 \to Aut(M_i)$ trivial for $i \geq n$;
\item if $M$ is a crossed complex with $\delta_1:M_1 \to M_0$ a crossed module central extension then $M$ belongs to $\mathcal{E}$.
\end{enumerate}
In particular, for $n=2$ this holds for the triplet:
\[ (Ker(\mathbf{cot}_{1}) , \X Mod, chn(Ab)_{\geq 2})\, . \]
\end{prop}

\begin{proof} 1) From Proposition \ref{ttcrs}, $\mu'_{n-1 \geq}=(Ker(\mathbf{cot}_{n-1}) ,Crs(Grp)_{n-1 \geq})$ is a torsion theory. It suffices to notice that the counit of $\mathbf{sk}_{n-1} \dashv \mathbf{tr}_{n-1}$ given by
\[\begin{tikzcd}  \dots\ar[r] &  0\ar[r]\ar[d] & 0\ar[r]\ar[d] & M_{n-1}\ar[r]\ar[d] & M_{n-2}\ar[r] \ar[d] & \dots \\ \dots\ar[r] &M_{n+1}\ar[r] & M_n\ar[r] & M_{n-1}\ar[r] & M_{n-2}\ar[r] & \dots  \end{tikzcd} \]
is monic since each component is an injective morphism.

2) It is clear that the pair $(Crs(Grp)_{n-1 \geq}, Crs(Grp)_{ \geq n})$ satisfies TT1 of the definition of a $\mathcal{E}$-torsion theory. Now, let $M$ be a crossed complex with trivial actions $M_0 \to Aut(M_i)$ and consider the morphisms $\mathbf{tr}_{n-1}(M) \to M \to \mathbf{tr'}_n(M)$ in $Crs(Grp)$:
\[\begin{tikzcd}  
\dots\ar[r] & 0\ar[r]\ar[d] & 0\ar[r]\ar[d] & M_{n-1}\ar[r]\ar[d] & M_{n-2}\ar[r] \ar[d] & \dots 
\\ \dots\ar[r] & M_{n+1}\ar[d] \ar[r] &  M_n\ar[r]\ar[d] & M_{n-1}\ar[r]\ar[d] & M_{n-2}\ar[r]\ar[d] & \dots 
\\ \dots\ar[r] & M_{n+1} \ar[r] & M_n\ar[r] & 0\ar[r] & 0\ar[r] & \dots
 \end{tikzcd} \]
recall that the morphism $M \to \mathbf{tr'}_n(M)$ is indeed a morphism in $Crs(Grp)$ since the actions are trivial. It is a short exact sequence in $Crs(Grp)$ since it is a short exact sequence as chain complexes and the forgetful functor is conservative. 

3) It follows from the definition of crossed complex that if $\delta_1$ is surjective the actions $\delta_1(M_1)=M_0\to Aut(M_i)$ are trivial.
\end{proof}

\subsection{An example of a splitting CTF theory in a semi-abelian category}

In an abelian category a torsion theory $(\T, \F)$ is called splitting if the torsion subobject $t(X)$ of $X$ is a direct summand. In a semi-abelian category  we will call a torsion theory $(\T, \F)$ \textit{splitting} if for every object $X$ the associated exact sequence splits:
\[\begin{tikzcd}  0\ar[r] & t(X)\ar[r] & X\ar[r, shift right] & X/t(X)\ar[r]\ar[l, shift right]&0 \end{tikzcd}. \] 

In $RMod$ the category of modules over the ring $R$, a central idempotent element of $R$ induces a splitting torsion theory $(\T, \F)$ (also called centrally splitting), and even yields a TTF theory $(\F, \T, \F)$. Connections of splitting torsion theories and TTF theories are studied in \cite{jans65}.

The torsion theories  in simplicial groups and crossed complexes presented earlier are not splitting. However, we give an example of a splitting torsion theory in a semi-abelian category.

\begin{exmp}\label{exsplit} Let $KHopf_{coc}$ the category of cocommutative Hopf alegras over the field $K$ of characteristic 0. In \cite{GKV16}, the category $KHopf_{coc}$ is proved that it is a semi-abelian category and have a torsion theory $(KLie, Grp)$ where the $K$-Lie algebras are consider as the primitive Hopf $K$-algebras and $Grp$ is considered as the category of group Hopf $K$-algebras. Indeed, for every $K$-algebra $H$ the associated short exact sequence is given by the Cartier-Gabriel-Milnor-Moore-Kostant theorem:  there is a split short exact sequence
\[\begin{tikzcd}  0\ar[r] & U(L_H) \ar[r] & H \ar[r, shift right] & K[G_H]\ar[l, shift right] \ar[r] &0 \end{tikzcd} \]
where $K[G_H]$ is the group algebra of the group-like elements $G_H$ of $H$ and $U(L_H)$ is the enveloping algebra of the primitive elements $L_H$ of $H$.

In \cite{GKV18}, it is proved that the funtor $\mathcal{G}:KHopf_{coc}\to Grp$ that takes the group-like elements and $K[\_]: Grp \to KHopf_{coc}$ yield the adjunctions $\mathcal{G}\dashv K[\_]$ and $K[\_]\dashv \mathcal{G}$. As a consequence, $(KLie, Grp)$ is a splitting CTF theory.
\end{exmp}   

\section*{Acknowledgements}
This work is part of the author's Ph.D. thesis \cite{Lop22}. The author would like to thank his advisor Marino Gran for his guidance and suggestions, in particular for mentioning the example \ref{exsplit}.

{\footnotesize INSTITUT DE RECHERCHE EN MATH\'EMATIQUE ET PHYSIQUE,UNIVERSIT\'E CATHOLIQUE DE LOUVAIN, CHEMIN DU CYCLOTRON 2, 1348 LOUVAIN-LA-NUEVE, BELGIUM
\newline \textit{E-mail address}: guillermo.lopez@uclouvain.be, glopezcafaggi@gmail.com}

\end{document}